\documentclass[a4paper,11pt]{article}

\usepackage{amsmath}
\usepackage{amsthm}
\usepackage{amssymb}
\usepackage{cancel}
\usepackage{color}
\usepackage{calrsfs}
\usepackage{varioref}
\usepackage{bbm}
\definecolor{dblue}{rgb}{0.09,0.32,0.44} 
\usepackage[pdfborder={0 0 0}, colorlinks=true, pdfborderstyle={}, linkcolor=dblue, citecolor=dblue, urlcolor=blue]{hyperref}
\usepackage{epsfig}
\usepackage{psfrag}
\usepackage{subfigure}
\usepackage{graphicx}
\usepackage[mathscr,mathcal]{eucal}
\usepackage[shortlabels]{enumitem}
\usepackage{t1enc}
\usepackage[latin2]{inputenc}

\newtheorem {theorem}{Theorem}
\newtheorem {lemma}{Lemma}
\newtheorem {corollary}{Corollary}
\newtheorem {proposition}{Proposition}

\theoremstyle{remark}

\def \A {\mathbb A}

\def \D {\mathbb D}

\def \N {\mathbb N}

\def \R {\mathbb R}

\def \Z {\mathbb Z}

\def\boP{\mathbf{P}}

\def\cF{\mathcal{F}}

\def\cS{\mathcal{S}}

\def\cU{\mathcal{U}}
\def\cV{\mathcal{V}}

\def\cX{\mathcal{X}}

\def\vareps{\varepsilon}
\def \eps {\epsilon}

\newcommand{\probab}[1]{\ensuremath{\mathbf{P}\big(#1\big)}}
\newcommand{\expect}[1]{\ensuremath{\mathbf{E}\big(#1\big)}}

\newcommand{\condprobab}[2]{\ensuremath{\mathbf{P}\big(#1\bigm|#2\big)}}

\newcommand{\condexpect}[2]{\ensuremath{\mathbf{E}\big(#1\bigm|#2\big)}}

\newcommand{\ind}[1]{\ensuremath{\mathbbm{1}_{\{#1\}}}}

\def\clap#1{\hbox to 0pt{\hss#1\hss}}

\def\ueps{\underline{\epsilon}}

\def\ordo{o}

\def\one{\ensuremath\mathbbm{1}}

\DeclareMathOperator{\grad}{grad}

\newcommand{\abs}[1]{\ensuremath\left|{#1}\right|}

\def \wt {\widetilde}

\def\wh{\widehat}

\begin{document}

\title{Invariance Principle for the Random Lorentz Gas -- Beyond the Boltzmann-Grad Limit}

\author{
{\sc Christopher Lutsko$^*$ and B\'alint T\'oth$^{*\dagger}$}
\\[8pt]
$^*$University of Bristol, UK
\\
$^\dagger$R\'enyi Institute, Budapest, HU
}

\maketitle


\begin{center}
{\large\sl Dedicated to Oliver Penrose on his 91st birthday. }
\end{center}

\bigskip

\begin{abstract}
\noindent
We prove the invariance principle for a \emph{random Lorentz-gas} particle in 3 dimensions under the Boltzmann-Grad limit and simultaneous diffusive scaling. That is, for the trajectory of a point-like particle moving among infinite-mass, hard-core, spherical scatterers of radius $r$, placed according to a Poisson point process of density $\varrho$, in the limit $\varrho\to\infty$, $r\to0$, $\varrho r^{2}\to1$ up to time scales of order $T=\ordo(r^{-2}\abs{\log r}^{-2})$.  To our knowledge this represents the first significant progress towards solving rigorously this problem in classical nonequilibrium statistical physics, since the groundbreaking work of Gallavotti (1969) \cite{gallavotti_69, gallavotti_70, gallavotti_99}, Spohn (1978) \cite{spohn_78, spohn_80} and Boldrighini-Bunimovich-Sinai (1983) \cite{boldrighini-bunimovich-sinai_83}. The novelty is that the diffusive scaling of particle trajectory and the kinetic (Boltzmann-Grad) limit are taken simultaneously. The main ingredients are a coupling of the mechanical trajectory with the Markovian random flight process, and probabilistic and geometric controls on the efficiency of this coupling.   

Similar results have been earlier obtained for the weak coupling limit of classical and quantum random Lorentz gas, by Komorowski-Ryzhik (2006) \cite{komorowski-ryzhik_06}, respectively, Erd\H os-Salmhofer-Yau (2007) \cite{erdos-salmhofer-yau_08, erdos-salmhofer-yau_07}. However, the following are substantial differences between our work and these ones: 
(1) 
The physical setting is different: low density rather than weak coupling. 
(2)
The method of approach is different: probabilistic coupling rather than analytic/perturbative.
(3) 
Due to (2), the time scale of validity of our diffusive approximation -- expressed in terms of the kinetic time scale -- is much longer and fully explicit. 



\end{abstract}

\section{Introduction}
\label{s: Introduction}

We consider the Lorentz gas with randomly placed spherical hard core scatterers in $\R^d$. That is, place spherical balls of radius $r$ and infinite mass centred on the points of a Poisson point process of intensity $\varrho$ in $\R^d$, where $r^d \varrho$ is sufficiently small so that with positive probability there is free passage out to infinity, and define  $t\mapsto X^{r,\varrho} (t)\in\R^d$ to be the trajectory of a point particle starting with randomly oriented unit velocity, performing free flight in the complement of the scatterers and scattering elastically on them.

A major problem in mathematical statistical physics is to understand the diffusive scaling limit of the particle trajectory 
\begin{align}
\label{scaling}
t\mapsto \frac{X^{r, \varrho}(Tt)}{\sqrt{T}}, 
\qquad
\text{ as }
\qquad
T\to\infty.
\end{align}
Indeed, the \emph{Holy Grail} of this field of research would be to prove the invariance principle (i.e. weak convergence to a Wiener process with nondegenerate variance) for the sequence of processes in \eqref{scaling} in either the quenched or annealed setting (discussed in section \ref{ss: The random Lorentz gas}). For extensive discussion and historical background see the surveys \cite{spohn_80, dettmann_14, marklof_14} and the monograph \cite{spohn_91}. 

The same problem in the periodic setting, when the scatterers are placed in a periodic array and randomness comes only with the initial conditions of the moving particle, is much better understood, due to the fact that in the periodic case the problem is reformulated as diffusive limit of particular additive functionals of billiards in compact domains and thus heavy artillery of hyperbolic dynamical systems theory is efficiently applicable. In order to put our results in context, we will summarise very succinctly the existing results, in section \ref{ss: Summary of related work}.  

There has been, however, no progress in the study of the \emph{random} Lorentz gas informally described above, since the ground-breaking work of Gallavotti \cite{gallavotti_69, gallavotti_70, gallavotti_99}, Spohn \cite{spohn_78, spohn_80} and Boldrighini-Bunimovich-Sinai \cite{boldrighini-bunimovich-sinai_83} where weak convergence of the process $t\mapsto X^{r,\varrho}(t)$ to a continuous time random walk $t\mapsto Y(t)$ (called Markovian flight process) was established in the Boltzmann-Grad (a.k.a. low density) limit $r\to0$, $\varrho\to\infty$, $r^{d-1}\varrho \to 1$, in compact time intervals $t\in[0,T]$, with $T<\infty$, in the annealed \cite{gallavotti_69, gallavotti_70, gallavotti_99, spohn_78, spohn_80}, respectively, quenched \cite{boldrighini-bunimovich-sinai_83} setting. 

Our main result (see Theorem \ref{thm:tricky} in subsection \ref{ss: Main result}) proves the invariance principle in the annealed setting if we take the \emph{Boltzmann-Grad and diffusive limits simultaneously}: $r\to0$, $\varrho\to\infty$, $r^{d-1}\varrho \to 1$ \emph{and} $T=T(r) \to \infty$. Thus while the diffusive limit \eqref{scaling} with  fixed $r$ and $\varrho$ remains open, this is the first result proving convergence for times growing to infinity as $r \to 0$ in the setting of  randomly placed  scatterers, and hence it is a significant step towards the full resolution of the problem in the annealed setting. 

\subsection{The random Lorentz gas}
\label{ss: The random Lorentz gas}

We define now more formally the random Lorentz process. Place spherical balls of radius $r$ and infinite mass centred on the points of a Poisson point process of intensity $\varrho$ in $\R^{d}$, and define the trajectory $t\mapsto X^{r,\varrho} (t)\in\R^{d}$ of a particle moving among these scatterers as follows: 

\begin{enumerate}[-]

\item
If the origin is covered by a scatterer then $X^{r,\varrho}(t)\equiv 0$.

\item
If the origin is not covered by a scatterer then $t\mapsto X^{r,\varrho} (t)$ is the trajectory of a point-like particle starting from the origin with random velocity sampled uniformly from the unit sphere $S^{d-1}$ and flying with constant speed between successive elastic collisions on any one of the fixed, infinite mass scatterers.

\end{enumerate} 

\noindent
The randomness of the trajectory $t\mapsto X^{r,\varrho} (t)$ (when not identically $0$) is due to two sources: the random placement of the scatterers and the random choice of initial velocity of the moving particle. Otherwise, the dynamics of the moving particle is fully deterministic, governed by classical Newtonian laws. With probability 1 (with respect to both sources of randomness) the trajectory $t\mapsto X^{r,\varrho} (t)$ is well defined. 

Due to elementary scaling and percolation arguments
\begin{align}
\label{notrap}
\probab{ \text{the moving particle is not trapped  in a compact domain} } 
=
\vartheta_d(\varrho r^{d}), 
\end{align}
where $\vartheta_d:\R_+\to[0,1]$ is a percolation probability which is 
(i) 
monotone non-increasing; 
(ii) 
continuous except for one possible jump at a positive and finite critical value $u_c=u_c(d)\in(0 , \infty)$; 
(iii)
vanishing for $u\in(u_c, \infty)$ and positive for $u\in(0,u_c)$; 
(iv) 
$\lim_{u\to0}\vartheta_d(u)=1$. 
We assume that $\varrho r^{d}<u_c$. In fact, in the Boltzmann-Grad limit considered in this paper (see \eqref{BGlimit} below) we  will have  $\varrho r^{d}\to0$.

As discussed above, the Holy Grail of this field is a mathematically rigorous proof of the  invariance principle of the processes \eqref{scaling} in either one of the following two settings. 

\begin{enumerate}

\item[(Q)]
\emph{Quenched limit}: 
For almost all (i.e. typical) realisations of the underlying Poisson point process, with  averaging over the random initial velocity of the particle. In this case, it is expected that the variance of the limiting Wiener process is deterministic, not depending on the realisation of the underlying Poisson point process.

\item[(AQ)]
\emph{Averaged-quenched} (a.k.a. \emph{annealed}) \emph{limit}: 
Averaging over the random initial velocity of the particle \emph{and} the random placements of the scatterers.

\end{enumerate}

\medskip
\noindent
{\bf Remarks on the Hamiltonian character of the problem:}
We use a probabilistic language and setting in this paper, previously much of the literature has chosen to work in the Hamiltonian setting \cite{gallavotti_69, gallavotti_70, gallavotti_99, spohn_78, spohn_80}. However, we should emphasise that this probabilistic description is equivalent to the annealed setting of a Hamiltonian system: The Lorentz particle moves according to Newton's Second Law in the potential field of spherical hard core scatterers centred in the points of a Poisson Point Process. The potential field is 
\[
\Phi^{r,\varrho}(x)
:=
\sum_{q\in \omega_\varrho} \varphi((x-q)/r),
\]
where $\omega_\varrho$ is the realisation of a Poisson Point Process of intensity $\varrho$ in $\R^d\setminus B_{0,r}$ (that is, no scatterer within distance $r$ from the origin) and $\varphi(x):= \infty \ind{\abs{x}<1}$ is a spherical hard-core potential. The Hamiltonian equations of motion of the Lorentz particle are \emph{formally} written as follows  
\[
\dot X^{r,\varrho}(t)= V^{r,\varrho}(t), 
\qquad
\dot V^{r,\varrho}(t)= - \grad \Phi^{r,\varrho}(X^{r,\varrho}(t)), 
\]
with initial conditions
\[
X^{r,\varrho}(0)=0,
\qquad 
V^{r,\varrho}(0)\in S^{d-1}.
\]
However, since the interaction potential is \emph{hard core}, the equations of motion are singular and should be taken with a grain of salt.

\subsection{The Boltzmann-Grad limit}
\label{ss: The Boltzmann-Grad limit}

The Boltzmann-Grad limit is the following low (relative) density limit of the scatterer configuration: 
\begin{align}
\label{BGlimit}
r\to0, 
\qquad\qquad
\varrho\to\infty, 
\qquad\qquad
\varrho r^{d-1}\to v_{d-1}, 
\end{align}
where $v_{d-1}$ is the area of the $(d-1)$-dimensional unit disc. In this limit the expected free path length between two successive collisions will be 1. Other choices of $\lim \varrho r^{d-1}\in (0,\infty)$ are equally legitimate and would change the limit only by a time (or space) scaling factor. 

It is not difficult to see that in the averaged-quenched setting and under the Boltzmann-Grad limit \eqref{BGlimit} the distribution of the first free flight length starting at any deterministic time, converges to an $EXP(1)$ and the jump in velocity after the free flight happens in a Markovian way with transition kernel
\begin{align}
\label{transition}
\condprobab{v_{\texttt{out}}\in dv^\prime}{v_{\texttt{in}}=v}
=
\sigma(v,v^{\prime})
dv^\prime,
\end{align}
where $dv^\prime$ is the surface element on $S^{d-1}$ and $\sigma: S^{d-1}\times S^{d-1} \to \R_+$ is the normalised \emph{differential cross section} of a spherical hard core scatterer, computable as
\begin{align}
\label{dcs}
\sigma(v,v^{\prime})
=
\frac{1}{4v_{d-1}}
\abs{v-v^\prime}^{3-d}.
\end{align}  
Note that in $3$-dimensions the transition probability \eqref{transition} of velocity jumps is uniform. That is, the outgoing velocity $v_{\tt{out}}$ is uniformly distributed on $S^{2}$, independently of the incoming velocity $v_{\tt{in}}$.

It is intuitively compelling but far from easy to prove that under the Boltzmann-Grad limit \eqref{BGlimit} 
\begin{align}
\label{BGprocesslimit}
\Big\{t\mapsto X^{r,\varrho}(t)\Big\} 
\Rightarrow 
\Big\{t\mapsto Y(t)\Big\}, 
\end{align}
where the symbol $\Rightarrow$ stands for weak convergence (of probability measures) on  the space of continuous trajectories in $\R^{d}$, see \cite{billingsley_68}. The process $t\mapsto Y(t)$ on the right hand side is the Markovian random flight process  consisting of independent free flights of $EXP(1)$-distributed length, with Markovian velocity changes according to the scattering transition kernel \eqref{transition}. A formal construction of the process $t\mapsto Y(t)$ is given in section \ref{ss: Ingredients and the Markovian flight process}. The limit \eqref{BGprocesslimit}, valid in any compact time interval $t\in[0,T]$, $T<\infty$, is rigorously established in the averaged-quenched setting in \cite{gallavotti_69, gallavotti_70, gallavotti_99, spohn_78, spohn_80},  and in the quenched setting in \cite{boldrighini-bunimovich-sinai_83}. In \cite{spohn_78} more general point processes of the scatterer positions, with sufficiently strong mixing properties are considered.

The limiting Markovian flight process $t\mapsto Y(t)$ is a continuous time random walk. Therefore, by taking a second, diffusive limit \emph{after} the Boltzmann-Grad limit \eqref{BGprocesslimit}, Donsker's theorem (see \cite{billingsley_68}) yields indeed the invariance principle,
\begin{align}
\label{ipforY}
\Big\{t\mapsto T^{-1/2}Y(Tt)\Big\}
\Rightarrow 
\Big\{t\mapsto W(t)\Big\}, 
\end{align}  
as $T\to\infty$, where $t\mapsto W(t)$ is the isotropic Wiener process in $\R^{d}$ of non-degenerate variance. The variance of the limiting Wiener process $W$ can be explicitly computed but its concrete value has no importance. 

The natural question arises whether one could somehow interpolate between the double limit of taking first the Boltzmann-Grad limit \eqref{BGprocesslimit} and then the diffusive limit \eqref{ipforY} and the plain diffusive limit for the Lorentz process, \eqref{scaling}. Our main result, Theorem \ref{thm:tricky} formulated in section \ref{ss: Main result} gives a positive partial answer in dimension 3. Since our results are proved in three-dimensions from now on we formulate all statements in $d=3$ rather than general dimension. However, in some comments we will refer to general dimension $d$, when appropriate.

\subsection{Results}
\label{ss: Main result}

In the rest of the paper we assume $\varrho=\varrho(r)= \pi r^{-2}$ and drop the superscript $\varrho$ from the notation of the Lorentz process.   

Our results (Theorems \ref{thm: naive} and \ref{thm:tricky} formulated below) refer to a coupling -- joint realisation on the same probability space -- of the Markovian random flight process $t\mapsto Y(t)$, and the quenched-averaged (annealed) Lorentz process $t\mapsto X^r(t)$. The coupling is informally described later in this section and constructed with full formal rigour in section \ref{ss: The Lorentz exploration process}.

The first theorem states that in our coupling, up to time $T\ll r^{-1}$, the Markovian flight and Lorentz exploration processes stay together. 

\begin{theorem}
\label{thm: naive}
Let $T=T(r)$ be such that $\lim_{r\to0} T(r)=\infty$ and $\lim_{r\to0} rT(r)=0$. 
Then 
\begin{align}
\label{naive}
\lim_{r\to0}
\probab{\inf\{t: X^{r}(t)\not= Y(t)\}\le T }=0.
\end{align}
\end{theorem}

\medskip
\noindent
{\bf Remarks on Theorem \ref{thm: naive}:}
This result flashes some light on the strength of the probabilistic coupling method employed in this paper. In particular, with some elementary, purely probabilistic arguments it provides a formally stronger result than \cite{gallavotti_69, gallavotti_70, gallavotti_99, spohn_78} which state the weak limit \eqref{BGprocesslimit} (which follows from \eqref{naive}) for any fixed $T<\infty$. Note, however, that complementing the cited papers  with \emph{explicit error bounds} (which seems feasible) would give Theorem \ref{thm: naive}. So, Theorem \ref{thm: naive} on its own is a complement to these fundamental results. The full strength of our method is truly exhibit in Theorem \ref{thm:tricky}, our main result, which extends this result to time scales where nontrivial correlations already appear. However the proof of Theorem \ref{thm: naive} is included as it sheds light on the structure of the proof of Theorem \ref{thm:tricky}.

\begin{theorem}
\label{thm:tricky}
Let $T=T(r)$ be such that $\lim_{r\to0} T(r)=\infty$ and $\lim_{r\to0} r^2\abs{\log r}^2 T(r)=0$. 
Then, for any $\delta>0$,  
\begin{align}
\label{tricky-close}
\lim_{r\to0}
\probab{\sup_{0\le t\le T}\abs{X^{r}(t)-Y(t)}> \delta \sqrt{T} }=0, 
\end{align}
and hence 
\begin{align}
\label{ourip}
\Big\{t\mapsto T^{-1/2}X^{r}(Tt)\Big\} 
\Rightarrow 
\Big\{t\mapsto W(t)\Big\},
\end{align}
as $r\to 0$, in  the averaged-quenched sense. On the right hand side of \eqref{ourip} $W$ is a standard Wiener process of variance $1$ in $\R^3$. 
\end{theorem}

Indeed, the invariance principle \eqref{ourip} readily follows from the invariance principle for the Markovian flight process, \eqref{ipforY}, and the closeness of the two processes quantified in \eqref{tricky-close}. So, it  remains to prove \eqref{tricky-close}. This will be the content of the larger part of this paper, sections \ref{s: Beyond Naive}-\ref{s: End of proof}.  

The point of Theorem \ref{thm:tricky} is that the Boltzmann-Grad limit of scatterer configuration \eqref{BGlimit} and the diffusive scaling of the trajectory are done simultaneously, and not consecutively. The memory effects due to recollisions and shading are controlled up to the time scale $T=T(r)= \ordo(r^{-2}\abs{\log r}^{-2})$. 

\medskip
\noindent
{\bf Remarks on dimension:}
Our proof of Theorem \ref{thm:tricky} as it stands is valid in dimension $d=3$ only. We give here some comments on this fact and some hints on what can/could be proved by appropriate extensions of our method. However, we stress that any of these extensions would require some extra technical efforts. In order to keep the length of this paper under a reasonable limit, we do not include these arguments and extensions. 

\begin{enumerate}[(1)]

\item 
Issues in dimension $d=2$:
\label{2dim-comment}
\begin{enumerate}[(a)]
\item
  Probabilistic estimates at the core of our proofs are valid (as stated and used) only in the transient dimensions of random walk, $d\ge3$. This difference is implicit in the Green's function estimates of sections \ref{ss: Bounds - Naive} and \ref{ss: Bounds for Z}. Nevertheless, with extra effort and the cost of an  extra logarithmic factor (of order $\abs{\log r}$) the estimates in section \ref{ss: Bounds - Naive},  used in the proof of Theorem \ref{thm: naive}, can be saved in $d=2$, as well. Using these estimates (and relying on the Doeblin argument as hinted at in comment \ref{doeblin-comment} below, Theorem \ref{thm: naive}, with a lorgarithmic factor, i.e. with $T(r) =\ordo(r^{-1}\abs{\log r}^{-1})$ turns out to be valid in $d=2$, as well. 

\item
  A subtle geometric argument which will show up in sections \ref{ss: Geometric estimates}-\ref{ss: Proof of Corollary of main-geom} below, is valid only in $d\ge 3$, as well. This is unrelated to the recurrence/transience dichotomy and it is crucial in controlling the short range recollision and shadowing events, in the proof of Theorem \ref{thm:tricky}. 

\end{enumerate}

\item
\label{doeblin-comment}
The fact that in $d=3$ the differential cross section of hard spherical scatterers is uniform on $S^{d-1}$, c.f. \eqref{transition}, \eqref{dcs}, facilitates our arguments, since, in this case, the successive velocities of the random flight process $Y(t)$ form an i.i.d. sequence. In dimensions $d\not=3$ this is not the case. However, this is a technical issue only, not of crucial importance in the argument. In dimensions $d>3$ the differential cross section \eqref{dcs} satisfies Doeblin's condition $\inf_{v^{\prime}\in S^{d-1}} \sigma(v,v^{\prime}) >0$ and, using Doeblin's subtle trick, the sequence of successive velocities $(u_n)_{n\geq1}$ can be broken up in random i.i.d. blocks of exponentially tight lengths. This way, the main probabilistic steps of proof can be saved. In dimension $d=2$ Doeblin's condition does not hold directly, see \eqref{dcs}. However, it holds for the second convolution power $\sigma^2(v,v^{\prime}):=\int_{S^{d-1}}\sigma(v,v^{\prime\prime})\sigma(v^{\prime\prime},v^{\prime}) dv^{\prime\prime}$ (that is, for the conditional distribution of velocity after two consecutive scatterings). In this way the sequence of successive velocities $(u_n)_{n\geq1}$  can be broken up in random 1-dependent strongly stationary (rather than i.i.d.) blocks of exponentially tight lengths. The necessary bounds can be proved with the use of Green's function estimates for random walks with 1-dependent strongly stationary (rather than i.i.d.) steps. 

\item
Possible relation with singularity of the diffusion coefficient at $r\ll1$ and certain limitations of our method:
\\
We state without proof the following estimates: In any dimension $d\geq3$, there exist constants $C_k$, $k\geq1$, such that  
\begin{align*}
\probab{\inf_{t>\tau_k} \abs{Y(t)-Y(0)}<r}
\leq
\begin{cases}
C_k r^k 
&
\text{ if \ } k\in[1, d-2], 
\\
C_{d-1} r^{d-1}\abs{\log r} 
&
\text{ if \ } k=d-1, 
\\
C_k r^{d-1} &
\text{ if \ } k\in[d,\infty).
\end{cases}
\end{align*}
In plain words, these are bounds  on the probability of the continuous time random walk $t\mapsto Y(t)$ returning to the $r$-neighbourhood of its starting point, \emph{after $k$ or more scattering events}. As we are not going to use these bounds in a technical sense we don't prove them in this paper. The proof is not hard, however.  We present these bounds for the following two reasons: 

\begin{enumerate}[(a)]

\item
The logarithmic factor in the case $k=d-1$ seems to be related to the 
expected singularity of the diffusion coefficient in the presumed (but not proved) diffusive limit \eqref{scaling}, at $r\ll1$.

\item
The fact that for  $k\geq d$ the probability of recollision after $k$ scattering events is of order $r^{d-1}$, no matter how large $k$, is a clear warning about a limitation of our method, as is. Indeed, beyond time scales of order $T(r)\sim r^{-d+1}$ recollision patterns of all kinds of complexities occur, preventing any attempt of breaking up the time-line into quasi-independent legs, in a rigorously controlled way, as done in our proof. In conclusion, with hard work (in particular, hard geometric estimates) in dimension $d>3$ our proof could possibly be pushed up to time scales of order $T(r) =\ordo (r^{-d+1}\abs{\log r}^{-\alpha})$, with some $\alpha>0$, but certainly not further than this.  Our proof of Theorem \ref{thm:tricky} reaches essentially this limit, in $d=3$. Going to time scales longer than $r^{-d+1}$ would require some genuinely new idea. 

\end{enumerate}

\end{enumerate}

\medskip
\noindent
{\bf Remarks on robustness of the method:}
Our coupling method is robust, and could be applied to a variety of other interaction potentials with only technical and not conceptual extra difficulties. However, it does not seem to be easily extendable to point processes with correlations. 

\begin{enumerate}[(a)]

\item
Extending our methods to non-spherical hard-core scatterers would change the differential cross-section \eqref{dcs}. As such, the sequence of successive velocities of the $Y$-process would not be i.i.d.  but a genuine Markov chain. However there are probabilistic methods to handle such difficulties (e.g. using Doeblin's decomposition to independent blocks, as described in comment \ref{doeblin-comment} above). For example, we quote the invariance principle for Ehrenfest's wind-tree model (with hyper-cube scatterers), where - since the geometry is simpler - in a \emph{subsequent work} we prove a result analogous to Theorem \ref{thm:tricky} for times of order $T(r)=\ordo (r^{-2} \abs{\log r}^{-1})$, c.f. \cite{lutsko-toth_19}. 

\item
Extension to smooth potentials can be done as well, though this is somewhat trickier. In this case, besides changing the differential cross section \eqref{dcs} one should also deal with non-instantaneous interactions. This can be handled in the case of finite range smooth potentials. The coupled Markov process will be different: not simple flights with instantaneous velocity jumps but flights with sharp but smooth scatterings. For details of the realisation of this coupling see the forthcoming work \cite{helmuth-toth_20} where the weak coupling limit is pushed beyond the kinetic time scale with a similar, but not identical, probabilistic coupling method.

\item
In the construction of the exploration process - as a Markov process - it is essential, however, that the point process where the scatterers are centred be Poisson. Otherwise, the exploration process could not be realised as a Markov process and probably would be of not much use. (Recall that in \cite{spohn_78} the Boltzmann-Grad limit \eqref{BGprocesslimit} is proved for point processes with certain correlations allowed.)  This is certainly a limitation of our method. However, spatially inhomogeneous Poisson Point Processes could be handled. 

\end{enumerate}

\medskip
\noindent
{\bf Remarks on time scales:}
In various works the kinetic and diffusive limiting procedures are parametrised in different ways. We chose $r\to0$, $\varrho=r^{-1/(d-1)}\to\infty$. In order to gauge how far beyond the bare kinetic limit the diffusive limit is pushed, and to compare our time scale with existing results on weak coupling diffusive limits, cf \cite{komorowski-ryzhik_06, erdos-salmhofer-yau_08, erdos-salmhofer-yau_07} (see subsection \ref{ss: Summary of related work} below for some details), we should introduce the \emph{kinetic time scale} $T_{\mathrm{kin}}$. This is the space-time scale on which the kinetic limits  \cite{gallavotti_69, gallavotti_70, gallavotti_99, spohn_78, boldrighini-bunimovich-sinai_83, kesten-papanicolaou_80, erdos-yau_00, eng-erdos_05} hold, if formulated as scaling limit of the microscopic trajectory. In our notation it is   
\begin{align}
\label{kinetic time scale}
T_{\mathrm{kin}} := \varrho^{1/d}=r^{-(d-1)/d}. 
\end{align}
This time scale is the reference to which the time scale of validity of the diffusive limit should be gauged. In terms of the microscopic space-time - where typical spacing between scatterers is of order 1 and the Lorentz particle travels with velocity of order 1 - our diffusive limit holds for time scales up to 
\begin{align}
\label{diffusive time scale}
T_{\mathrm{diff}}
=
T_{\mathrm{kin}} T
\end{align}
with 
\begin{align}
\label{our time scale}
T 
= 
\ordo\left(r^{-2}\abs{\log r}^{-2}\right)
=
\ordo\left(T_{\mathrm{kin}}^{2d/(d-1)} (\log T_{\mathrm{kin}})^{-2}\right)
\buildrel{d=3}\over{=}
\ordo\left(T_{\mathrm{kin}}^{3} (\log T_{\mathrm{kin}})^{-2}\right).
\end{align}
This is to be compared with the time scales of the similar-in-spirit classical \cite{komorowski-ryzhik_06}, respectively, quantum \cite{erdos-salmhofer-yau_08, erdos-salmhofer-yau_07}, weak coupling diffusive limits, cf. \eqref{their timescale}. See subsection \ref{ss: Summary of related work} below for some details.

\bigskip

The proof of Theorems \ref{thm: naive} and \ref{thm:tricky} will be based on a coupling (that is: a joint realisation on the same probability space) of the Markovian flight process $t\mapsto Y(t)$ and the averaged-quenched realisation of the Lorentz process $t\mapsto X^r(t)$, such that the maximum distance of their positions up to time $T$ be small order of $\sqrt{T}$. The Lorentz process $t\mapsto X^r(t)$ is realised as an \emph{exploration} of the environment of scatterers. That is, as time goes on, more and more information is revealed about the position of the scatterers. As long as $X^r(t)$ traverses yet unexplored territories, it behaves just like the Markovian flight process $Y(t)$, discovering new, yet-unseen scatterers with rate 1 and scattering on them. However, unlike the Markovian flight process it has long memory, the discovered scatterers are placed forever and if the process $X^{r}(t)$ returns to these positions, recollisions occur. Likewise, the area swept in the past by the Lorentz exploration process $X^r(t)$ -- that is: a tube of radius $r$ around its past trajectory -- is recorded as a domain where new collisions can not occur. For a formal definition of the coupling see section \ref{ss: The Lorentz exploration process}. Let the associated velocity processes be $U(t):=\dot Y(t)$ and $V^{r}(t):=\dot X^r(t)$. These are almost surely piecewise constant jump processes. The coupling is realised in such a way, that 

\begin{enumerate} [(A)]

\item
\label{jointstart}
At the very beginning the two velocities coincide, $V^r(0)=U(0)$.

\item
\label{mismatches}
Occasionally, with typical frequency of order $r$ mismatches of the two velocity processes occur. These mismatches are caused by two possible effects: 

\begin{enumerate} [$\circ$]

\item
\emph{Recollisions} of the Lorentz exploration process with a scatterer placed in the past. This causes a collision event when $V^r(t)$ changes while $U(t)$ does not.

\item
Scatterings of the Markovian flight process $Y(t)$ in a moment when the Lorentz exploration process is in the explored tube, where it can not encounter a not-yet-seen new scatterer. In these moments the process $U(t)$ has a jump discontinuity, while the process $V^r(t)$ stays unchanged. We will call these events \emph{shadowed scatterings} of the Markovian flight process.  

\end{enumerate}

\item
\label{recouplings}
However, shortly after the mismatch events described in item \ref{mismatches} above, a new jointly realised  scattering event of the two processes occurs, recoupling the two velocity processes to identical values.  These recouplings occur typically at an $EXP(1)$-distributed time after the mismatches. 

\end{enumerate}


\begin{figure}[ht!]
  \begin{center}    
    \includegraphics[width=0.9\textwidth]{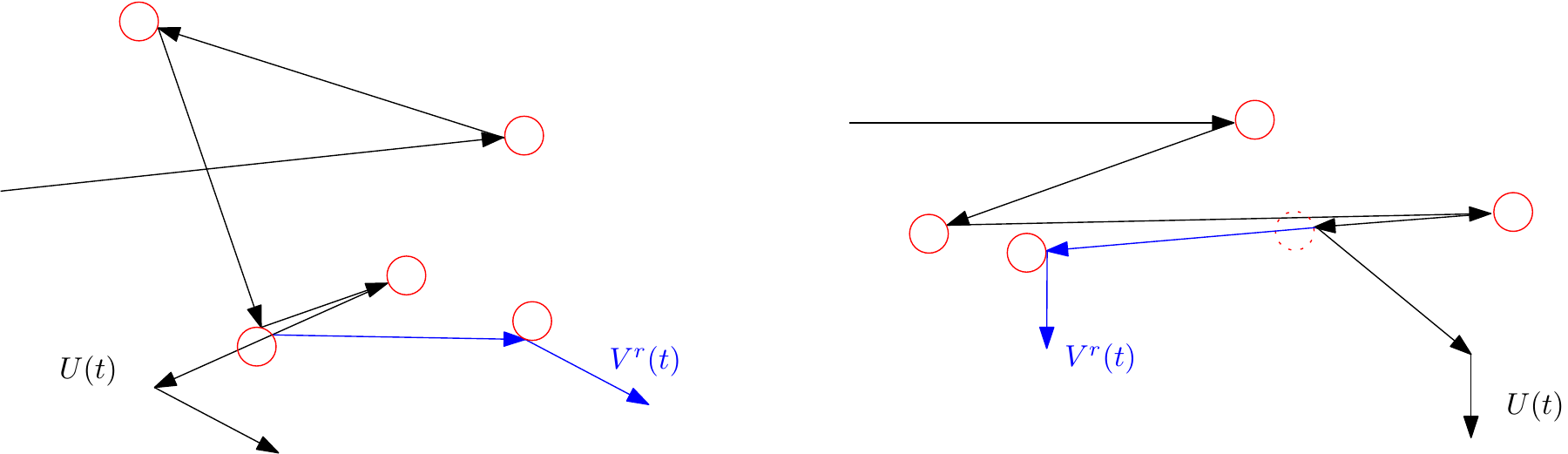}
  \end{center}
  \caption{%
    {\tt 
    The above image shows a recollision (left) and a shadowing event (right). Note that after each event $U$ and $V^r$ are no longer coupled. However at the next scattering, if possible, the velocities are recoupled. On the right-hand-side the virtual scatterer drawn in dotted line is shadowed. That is: it is physically not present in the mechanical trajectory.}
  }%
  
  \label{fig:stopping}

\end{figure}

\noindent
Summarising: The coupled velocity processes $t\mapsto (U(t), V^r(t))$ are realised in such a way that they assume the same values except for typical time intervals of length of order 1, separated by typical intervals of lengths of order $r^{-1}$. Other, more complicated mismatches of the two processes occur only at time scales of order $r^{-2}\abs{\log r}^{-2}$. If the probability of all mismatches, and the separation associated to those that do occur, can be controlled (this will be the content of the proof) then the following holds: 

\medskip
\noindent
Up to $T=T(r)=\ordo(r^{-1})$, with high probability there is no mismatch whatsoever between $U(t)$ and $V^{r}(t)$. That is, 
\begin{align}
\label{weaker}
\lim_{r\to0}
\probab{\inf\{t: V^{r}(t)\not=U(t)\}<T}
=
\lim_{r\to0}
\probab{\inf\{t: X^{r}(t)\not=Y(t)\}<T}
=0.
\end{align}
In particular, the invariance principle \eqref{ourip} also follows, with $T=T(r)=\ordo(r^{-1})$, rather than  $T=T(r)=\ordo(r^{-2}\abs{\log r}^{-2})$. As a by-product of this argument a new and handier proof of the theorem \eqref{BGprocesslimit} of \cite{gallavotti_69, gallavotti_70, gallavotti_99, spohn_78, spohn_80} also drops out. 

\medskip
\noindent
Going up to $T=T(r)=\ordo(r^{-2}\abs{\log r}^{-2})$  needs more argument. The ideas exposed in the outline \ref{jointstart}, \ref{mismatches}, \ref{recouplings} above lead to the following chain of bounds: 
\begin{align*}
\notag\max_{0\le t \le 1} 
\abs{\frac{X^{r}(Tt)}{\sqrt{T}} - \frac{Y(Tt)}{\sqrt{T}}}
=
&
\frac{1}{\sqrt{T}}\max_{0\le t\le 1}\abs{\int_0^{Tt} \left(V^r(s)-U(s)\right) ds}
\\
\le
&
\frac{1}{\sqrt{T}}\int_0^T \abs{ V^r(s)-U(s)} ds
\asymp
\frac{1}{\sqrt{T}} T r
=
\sqrt{T}r. 
\end{align*}
In the $\asymp$ step we use the arguments \ref{mismatches} and \ref{recouplings}. Finally, choosing in the end $T=T(r)=\ordo(r^{-2})$ we obtain a tightly close coupling of the diffusively scaled processes $t\mapsto X^{r}(Tt)/\sqrt{T}$ and $t\mapsto Y(Tt)/\sqrt{T}$, \eqref{tricky-close}, and hence the invariance principle \eqref{ourip}, for this longer time scale. This hand-waving argument should, however, be taken with a grain of salt: it does not show the logarithmic factor, which arises in the fine-tuning.  

\subsection{Summary of related work}
\label{ss: Summary of related work}

In order to put our results in context we succinctly summarise the related most important results in the mathematically rigorous treatment of diffusion in the Lorentz gas. As  Hendrik Lorentz's seminal paper \cite{lorentz_05} -- where he proposes the periodic setting of what we call today the Lorentz gas for modelling diffusion and transport in solids -- was published in 1905, and due to the large amount of work done in this field, we can not strive for exhaustion, and mention only a (possibly subjective) selection of the mathematically rigorous results. For more comprehensive historical overview we refer the reader to the survey papers \cite{dettmann_14, marklof_14, spohn_80} and the monograph \cite{spohn_91}.

\subsubsection*{Scaling limit of the periodic Lorentz gas}

As  already mentioned, diffusion in the periodic setting is much better understood than in the random setting. This is due to the fact that diffusion in the periodic Lorentz gas can be reduced to the study of limit theorems of some particular additive functionals of billiard flows in compact domains. Heavy tools of hyperbolic dynamics provide the technical arsenal for the study of these problems. 

The first breakthrough was the fully rigorous proof, by Bunimovish and Sinai \cite{bunimovich-sinai_80}, of the invariance principle (diffusive scaling limit) for the Lorentz particle trajectory in a two-dimensional periodic array of spherical scatterers  with finite horizon. (Finite horizon means that the length of the straight path segments not intersecting a scatterer is bounded from above.) This result was extended by Chernov \cite{chernov_94}, to higher dimensions, under a still-not-proved technical assumption on singularities of the corresponding billiard flow. 

In the case of infinite horizon (e.g. the plain $\Z^d$ arrangement of the spherical scatterers of diameter less than the lattice spacing) the free flight distribution of a  particle flying in a uniformly sampled random direction has a heavy tail which causes a different type of long time behaviour of the particle displacement. The arguments of Bleher \cite{bleher_92} indicated that in the two-dimensional case super-diffusive scaling of order $\sqrt{t\log t}$ is expected. For the Lorentz-particle displacement in the $2$-dimensional periodic case with infinite horizon, a central limit theorem with this anomalous scaling was proved with full rigour by Varj\'u and Sz\'asz \cite{szasz-varju_07} and Dolgopyat and Chernov \cite{dolgopyat-chernov_09}. The periodic infinite horizon case in dimensions $d\ge3$ remains open. 

\subsubsection*{Boltzmann-Grad limit of the periodic Lorentz gas}

The Boltzmann-Grad limit in the periodic case means spherical scatterers of radii $r\ll 1$ placed on the points of the hypercubic lattice $r^{(d-1)/d}\Z^d$. The particle starts with random initial position and velocity sampled uniformly and collides elastically on the scatterers. For a full exposition of the long and complex history of this problem we quote the surveys \cite{golse_06, marklof_14} and recall only the final, definitive results. 

In Caglioti-Golse \cite{caglioti-golse_08} and Marklof-Str\"ombergsson \cite{marklof-strombergsson_11} it is proved that in the Boltzmann-Grad limit the trajectory of the Lorentz particle in any compact time interval $t\in[0,T]$ with $T<\infty$ fixed, converges weakly to a non-Markovian flight process which has, however, a complete description in terms of a Markov chain of the successive collision impact parameters and, conditionally on this random sequence, independent flight lengths. (For a full description in these terms see \cite{marklof-toth_16}.) As a second limit, the invariance principle is proved for this non-Markovian random flight process, with superdiffusive scaling $\sqrt{t \log t}$, in Marklof-T\'oth \cite{marklof-toth_16}. Note that in this case the second limit doesn't just drop out from Donsker's theorem as it did in the random scatterer setting. The results of  \cite{caglioti-golse_08} are valid in $d=2$ while those of \cite{marklof-strombergsson_11}  and \cite{marklof-toth_16} in arbitrary dimension. 

Interpolating between the plain scaling limit in the infinite horizon case (open in $d\ge3$) and the kinetic limit, by simultaneously taking the Boltzmann-Grad limit \emph{and} scaling the trajectory by $\sqrt{T\log T}$, where $T=T(r)\to\infty$ with some rate, would be the problem analogous to our Theorem \ref{thm: naive} or Theorem \ref{thm:tricky}. This is widely open. 

\subsubsection*{The weak coupling limit}

The weak coupling is physically a different limiting procedure for obtaining diffusion of moving particle among fixed scatterers. In conformity with the usual notation of the weak coupling literature we will use the scaling parameter $\varepsilon\to0$. Infinite mass fixed scatterers are again placed on the points of a Poisson point process of density $\varrho=\varepsilon^{-d}$ in $\R^d$. However, now it is assumed that the compactly supported and spherically symmetric scattering potential $\cU$ of radius $r=\varepsilon$, centred at the scatterer positions, is smooth and bounded rather than hard core. Note that $\varrho=\varepsilon^{-d}$, $r=\varepsilon$ means just a linear spatial scaling by a factor $\varepsilon$. In this limit,  rather than scaling down excessively the radius of support, the strength of the potential is scaled. Newton's equations of motion for the kinetically scaled particle are
\begin{align*}
\dot X^{\varepsilon}(t) 
=
V^{\varepsilon}(t), 
\qquad
\dot V^{\varepsilon}(t) 
= 
-\nabla U^{\varepsilon} (X^{\varepsilon}(t)) 
\end{align*}
in the potential field 
\begin{align*}
U^{\varepsilon} (x)
=
\sum_{q\in\omega}
\varepsilon^{1/2}
\cU(\varepsilon^{-1}(x-q)),
\end{align*}
where $\omega$ is the realisation of the Poisson point process of intensity $\varrho=\varepsilon^{-d}$. 

From the work of Kesten and Papanicolaou \cite{kesten-papanicolaou_80} it follows that 
\begin{align}
V^{\varepsilon}(t)\Rightarrow \cV(t), 
\qquad
X^{\varepsilon}(t)\Rightarrow \cX(t):=\int_0^t\cV(s)ds, 
\end{align}
where the limiting  velocity process $\cV(t)$ is a homogeneous diffusion (i.e. Brownian motion) on the surface of $S^{d-1}$ and the weak convergence is meant in the space of continuous trajectories endowed with uniform topology on compact time intervals, cf \cite{billingsley_68}. See also the survey \cite{spohn_80}. Taking a second, diffusive limit, $T^{-1/2} \cX(Tt)\rightarrow W(t)$, the displacement process converges to Brownian motion, as $T\to\infty$. 

The simultaneous kinetic \emph{and} diffusive limit in this context is done by Komorowski and Ryzhik in \cite{komorowski-ryzhik_06} where it is proved that in dimension $d\ge3$, up to time scales
\begin{align}
\label{their timescale}
T=T(\varepsilon)=\varepsilon^{-\kappa}=T_{\mathrm{kin}}^{\kappa},
\qquad
\kappa\in(0,\kappa_0), 
\qquad
\kappa_0>0,
\end{align}
the diffusive limit 
\begin{align}
\label{komorowski-ryzhik-limit}
T^{-1/2}X^{\varepsilon}(Tt)
\Rightarrow W(t)
\end{align}
holds. In \eqref{their timescale} $\kappa_0$ is small (possibly, very small) and positive, its numerical value is not specified and difficult to determine from the various technical estimates.  

To our knowledge this was the first case when diffusive limit was rigorously established beyond the kinetic time scale in a context which includes the random Lorentz gas. We also note that the results in \cite{kesten-papanicolaou_80} and \cite{komorowski-ryzhik_06} are formulated in more general context of spatially ergodic random potential fields with regularity conditions assumed. This covers weak coupling of the random Lorentz gas as particular case. Our main Theorem \ref{thm:tricky} should be compared with this result. In particular, the time scale of validity of the diffusive limit \eqref{their timescale} is to be compared with the time scale  \eqref{our time scale} up to which our Theorem \ref{thm:tricky} is valid. 

In the forthcoming work \cite{helmuth-toth_20} the diffusive limit under weak coupling \eqref{komorowski-ryzhik-limit} is proved with probabilistic coupling method somewhat similar but not identical to the present one, for time scales $T(\varepsilon) = \ordo(\vareps^{-d+2} \abs{\log \varepsilon}^{-\alpha})$ in any $d\geq 3$, with some $\alpha<\infty$, improving thus considerably the result of \cite{komorowski-ryzhik_06}.

\subsubsection*{The quantum Lorentz gas}

The quantum versions of the weak coupling and low density limits for the random Lorentz gas were considered in Erd\H os-Yau \cite{erdos-yau_00}, respectively, Eng-Erd\H os \cite{eng-erdos_05}, where the long time evolution of a quantum particle interacting with a random potential is studied. It is proved that the phase space density of the quantum evolution converges weakly to the solution of the linear Boltzmann (or, Langevin) equation, with diffusive, respectively, hopping scattering kernels. These results are the quantum analogues of the classical (i.e. non-quantum) kinetic limits of \cite{kesten-papanicolaou_80} (for weak coupling), respectively, \cite{gallavotti_69, gallavotti_70, gallavotti_99, spohn_78, spohn_80} (for low density). 

In the weak coupling setup the simultaneous kinetic \emph{and} diffusive scaling limit, formally analogous to \cite{komorowski-ryzhik_06} was done by Erd\H os-Salmhofer-Yau \cite{erdos-salmhofer-yau_08, erdos-salmhofer-yau_07} where it is proved that under a scaling limit similar to \eqref{their timescale}, \eqref{komorowski-ryzhik-limit} the time evolution of the spatial density of the quantum particle weakly coupled with the fixed scatterers converges to the solution of the heat equation. In this case the numerical value of the upper bound on the scaling exponent $\kappa$ is specified in $d=3$ as $\kappa_0=1/370$ (see Theorem 2.2 in \cite{erdos-salmhofer-yau_08}). 

For a comprehensive survey of the kinetic and kinetic-diffusive limits in the quantum case see also \cite{erdos_12}. 

\subsubsection*{Miscellaneous}

Looking into the future: Liverani investigates the periodic Lorentz gas with finite horizon with local random perturbations in the cells of periodicity: a basic periodic structure with spherical scatterers centred on $\Z^d$ with extra scatterers placed randomly and independently within the cells of periodicity, \cite{aimino-liverani_18}. This is an interesting mixture of the periodic and random settings which could succumb to a mixture of dynamical and probabilistic methods, so-called deterministic  walks in random environment.

\subsection{Structure of the paper}
\label{ss: Structure of the paper}
 
The rest of the paper is devoted to the rigorous statement and proof of the arguments exposed in \ref{jointstart}, \ref{mismatches}, \ref{recouplings} above. Its overall structure is as follows: 

\begin{itemize}[-]

\item 
\emph{Section \ref{s: Construction}:} 
We construct the Markovian flight and Lorentz exploration processes and thus lay out the coupling argument which is essential moving forward. Moreover, we will also introduce an auxiliary process, $Z$, a  short-sighted or forgetful version of $X$ which somehow interpolates between the processes $Y$ and $X$.

\item 
\emph{Section \ref{s: Naive}:} 
We prove Theorem \ref{thm: naive}. We go through the proof of this statement as it is both informative for the dynamics, and the proof of Theorem \ref{thm:tricky} in its full strength will follow partially similar lines, however with substantial differences. 
  
\end{itemize}

Sections \ref{s: Beyond Naive}-\ref{s: End of proof} are fully devoted to the proof of Theorem \ref{thm:tricky}, as follows: 

\begin{itemize}[-] 

\item 
\emph{Section \ref{s: Beyond Naive}:} 
We break up the process $Z$ into independent legs of exponentially tight lengths. From here we state two propositions which are central to the proof. They state that
\\ 
(i)
with high probability the process $X$ does not differ from $Z$ in each leg; 
\\
(ii)
with high probability, the different legs of the process $Z$ do not interact (up to times of our time scales).

\item 
\emph{Section \ref{s: Proof of Proposition bw-legs}:} 
We prove the proposition concerning interactions between legs.

\item 
\emph{Section \ref{s: Proof of Proposition Z=X in one leg}:} 
We prove the proposition concerning coincidence, with high probability, of the processes $X$ and $Z$ within a single leg. This section is longer than the others, due to the subtle geometric arguments and estimates needed in this proof. 
   
\item 
\emph{Section \ref{s: End of proof}:} 
We finish off the proof of Theorem \ref{thm:tricky}.
    
\end{itemize}

\section{Construction}
\label{s: Construction}

\subsection{Ingredients and the Markovian flight process}
\label{ss: Ingredients and the Markovian flight process}

Let $\xi_j\in\R_+$ and $u_{j}\in\R^{3}$, $j=-2,-1,0,1,2,\dots$, be completely independent random variables (defined on an unspecified probability space $(\Omega, \cF, \boP)$) with distributions: 
\begin{align}
\label{seqingred}
\xi_j \sim EXP(1), 
\qquad
u_j\sim UNI(S^{2}),
\end{align}
and let 
\begin{align}
\label{ysteps}
y_j:= \xi_j u_j\in\R^{3}. 
\end{align}
For later use we also introduce the sequence of indicators
\begin{align}
\label{epsilon-ind}
\epsilon_j:=\one\{\xi_j<1\}, 
\end{align}
and the corresponding conditional exponential distributions $EXP(1|1):=\text{distrib}(\xi\,|\,\epsilon=1)$, respectively, $EXP(1|0)=\text{distrib}(\xi\,|\,\epsilon=0)$,  with distribution densities
\begin{align*}
(e-1)^{-1} e^{1-x} \one\{0\le x< 1\}, 
\qquad\text{ respectively, }\qquad
e^{1-x} \one\{1\le x< \infty\}. 
\end{align*} 
We will also use the notation $\ueps:=(\epsilon_j)_{j\ge 0}$ and call the sequence $\ueps$ the \emph{signature} of the i.i.d. $EXP(1)$-sequence $(\xi_j)_{j\ge 0}$.

The variables $\xi_j$ and $u_{j}$ will be, respectively,  the consecutive \emph{flight length/flight times} and \emph{flight velocities} of the Markovian flight process $t\mapsto Y(t)\in\R^{3}$ defined below. 

Denote, for $n\in\Z_+$, $t\in\R_+$, 
\begin{align}
\label{tau-and-nu-for-Y}
\tau_n:=\sum_{j=1}^n \xi_j,
\qquad\quad
\nu_t:=\max\{n:\tau_n\leq t\}, 
\qquad\quad
\{t\}:= t-\tau_{\nu_t}.
\end{align}
That is: 
$\tau_n$ denotes the consecutive scattering times of the flight process, 
$\nu_t$ is the number of scattering events of the flight process $Y$ occurring in the time interval $(0,t]$, 
and 
$\{t\}$ is the length of the last free flight before time $t$. 

Finally let 
\begin{align*}
Y_n
:=
\sum_{j=1}^{n} \xi_j u_{j}
=
\sum_{j=1}^{n} y_{j}
, 
\qquad\qquad
Y(t)
:=
Y_{\nu_t}+ \{t\}u_{\nu_t+1}.
\end{align*}

We shall refer to the process $t\mapsto Y(t)$ as the Markovian flight process. This will be our fundamental probabilistic object. All variables and processes will be defined in terms of this process, and adapted to the natural continuous time filtration $(\cF_t)_{t\ge0}$ of the flight process: 
\begin{align*}
\cF_t:=\sigma(u_0, (Y(s))_{0\le s\le t}). 
\end{align*} 

Note that the processes $n\mapsto Y_n$, $t\mapsto Y(t)$ and their respective natural filtrations $(\cF_n)_{n\ge0}$, $(\cF_t)_{t\ge0}$, do not depend on the parameter $r$.

We also define, for later use, the \emph{virtual scatterers} of the flight process $t\mapsto Y(t)$. Let 
\begin{align*}
\begin{aligned}
Y^{\prime}_{k}
&:=
Y_k + r \frac{u_{k}-u_{k+1}}{\abs{u_{k}-u_{k+1}}}
=
Y_k + r \frac{\dot Y(\tau_k^-)-\dot Y(\tau_k^+)}{\abs{\dot Y(\tau_k^-)-\dot Y(\tau_k^+)}},
&&
\qquad
k\ge0, 
\\
\cS^Y_n
&:=
\{Y^{\prime}_k\in\R^{3}: 0\le k\le n\},
&&
\qquad
n\ge0.
\end{aligned}
\end{align*} 
Here and throughout the paper we use the notation $f(t^\pm):=\lim_{\vareps\downarrow0}f(t\pm\varepsilon)$. 
\\ 
The points $Y^{\prime}_n\in\R^{3}$ are the centres of virtual spherical scatterers of radius $r$ which \emph{would have caused} the $n$th scattering event of the flight process. They do not have any influence on the further trajectory of the flight process $Y$, but will play role in the forthcoming couplings.

\subsection{The Lorentz exploration process}
\label{ss: The Lorentz exploration process}

Let $r>0$, and $\varrho=\varrho(r)=\pi r^{-2}$. 
We define the Lorentz exploration process $t\to X(t)=X^{r}(t)\in\R^{3}$, coupled with the flight process $t\mapsto Y(t)$, adapted to the filtration $\left(\cF_t\right)_{t\ge0}$. The process $t\mapsto X(t)$ and all upcoming random variables related  to it \emph{do depend} on the choice of the parameter $r$ (and $\varrho$), but from now on we will suppress explicit notation of dependence upon these parameters.

The construction goes inductively, on the successive time intervals $[\tau_{n-1}, \tau_{n})$, $n=1,2, \dots$. Start with \ref{stepone} and then  iterate indefinitely \ref{steptwo} and \ref{stepthree} below.

\begin{enumerate} [{[}Step 1{]}]

\item
\label{stepone}
Start with 
\begin{align*}
X(0)=X_0=0, 
\qquad
V(0^+)=u_1, 
\qquad
X^{\prime}_0:=r \frac{u_{0}-u_{1}}{\abs{u_{0}-u_{1}}},
\qquad
\cS^X_0=\{X^{\prime}_0\}.
\end{align*}
Note that the trajectory of the exploration process $X$ begins with a collision at time $t=0$. This is not exactly as described previously but is of no consequence and aids the later exposition.

{\tt Go to \ref{steptwo}.}

\item
\label{steptwo}
This step starts with given $X(\tau_{n-1})=X_{n-1}\in \R^{3}$, $V(\tau_{n-1}^+)\in S^{2}$ and $\cS^X_{n-1}=\{X^{\prime}_k: 0\le k \le n-1\}\subset \R^{3}\cup\{\bigstar\}$, where

\begin{enumerate}[$\circ$]
\item
$\bigstar$ is a fictitious point at infinity, with $\inf_{x\in\R^{3}}\abs{x-\bigstar}=\infty$, introduced for bookkeeping reasons;

\item
$\abs{X_{n-1}- X^{\prime}_k}\in (r, \infty]$ for $0\le k<n-1$, and $\abs{X_{n-1}- X^{\prime}_{n-1}}\in\{r, \infty\}$.

\end{enumerate}

(Note that, due to absolute continuity of the flight time distribution $\abs{X_{l}- X^{\prime}_k}\not= r$, for $k\not=l$, with probability 1.)

The trajectory $t\mapsto X(t)$, $t\in[\tau_{n-1}, \tau_{n})$, is defined as free motion with elastic collisions on fixed spherical scatterers of radius $r$ centred at the points in $\cS^X_{n-1}$. At the end of this time interval the position and velocity of the Lorentz exploration process are $X(\tau_n)=:X_n$, respectively, $V(\tau_n^-)$. 

{\tt Go to \ref{stepthree}.}

\item
\label{stepthree}
Let 
\begin{align*}
X^{\prime\prime}_{n}
:=
X_n + r \frac{V(\tau_{n}^-)-u_{n+1}}{\abs{V(\tau_{n}^-)-u_{n+1}}},
\qquad
d_{n}
:=
\min_{0\le s< \tau_{n}} \abs{ X(s)-  X^{\prime\prime}_{n}}.
\end{align*} 
Note that $d_n\le r$.
\begin{enumerate}[$\circ$]
\item
If $d_{n}<r$ then let
$X^{\prime}_n:=\bigstar$, and  
$V(\tau_{n}^+)=V(\tau_{n}^-)$. 
\item
If $d_{n}= r$ then let
$X^{\prime}_n:=X^{\prime\prime}_n$, and  
$V(\tau_{n}^+)=u_{n+1}$. 
\end{enumerate}
\noindent
Set $\cS^X_{n}=\cS^X_{n-1} \cup \{X^{\prime}_{n}\}$. 

{\tt Go back to \ref{steptwo}.}

\end{enumerate}

\noindent
The process $t\mapsto X(t)$ is indeed adapted to the filtration $(\cF_t)_{0\le t<\infty}$ and indeed has the averaged-quenched distribution of the Lorentz process. This follows from the fact that the scatterers of the Lorentz process are centred on a Poisson Point Process and thus when sweeping not-yet-seen areas no information from the past interferes.

Our notation is fully consistent with the one used for the Markovian process $Y$: $X_n:=X(\tau_n)$ and 
\begin{align*}
X^{\prime}_{k}
&:=
\begin{cases}
\displaystyle
X_k + r \frac{\dot X(\tau_k^-)-\dot X(\tau_k^+)}{\abs{\dot X(\tau_k^-)-\dot X(\tau_k^+)}}
&\text{ if }\quad
\dot X(\tau_k^-)\not=\dot X(\tau_k^+), 
\\
\bigstar
&\text{ if }\quad
\dot X(\tau_k^-)=\dot X(\tau_k^+),
\end{cases}
&&
\qquad
k\ge0, 
\\
\cS^X_n
&:=
\{X^{\prime}_k\in\R^{3}\cup \{\bigstar\}: 0\le k\le n\},
&&
\qquad
n\ge0.
\end{align*} 
The second alternative above happens when the scattering event offered by the Poisson flaw is suppressed due to shadowing.

\subsection{Mechanical consistency and compatibility of piece-wise linear trajectories in $\R^3$}
\label{ss: Mechanical consistency and compatibility}

The key notion in the exploration construction of section \ref{ss: The Lorentz exploration process} was mechanical $r$-consistency, and $r$-compatibility of finite segments of piece-wise linear trajectories in $\R^3$, which we are going to formalise now, for later reference. 
We will apply the notion of $r$-consistency/compatibility to several different process in what follows. Thus we let $\{t \mapsto \mathscr{Z}(t)\}$ denote any one of the aforementioned/forthcoming piece-wise linear processes.

Thus, let
\begin{align*}
n\in\N, 
\qquad
\tau_0\in\R,
\qquad
\mathscr{Z}_0\in\R^3, 
\qquad
v_0,\dots, v_{n+1} \in S^2
\qquad
t_1,\dots,t_n\in\R_+,
\end{align*}
be given and define for $j=0, \dots, n-1$,
\begin{align*}
\tau_j:=
\tau_0+\sum_{k=1}^j t_k, 
\qquad
\mathscr{Z}_j:=
\mathscr{Z}_0+\sum_{k=1}^j t_kv_k,
\qquad
\mathscr{Z}_j^{\prime}
:= 
\begin{cases}
\displaystyle 
\mathscr{Z}_j+ r \frac{v_j-v_{j+1}}{\abs{v_j-v_{j+1}}}
&\text{ if }v_j\not=v_{j+1}, 
\\
\bigstar
&\text{ if }v_j=v_{j+1}, 
\end{cases}
\end{align*}
and for $t\in[\tau_j, \tau_{j+1}]$, $j=0, \dots, n$,

\begin{align*}
\mathscr{Z}(t)
:=
\mathscr{Z}_j+ (t-\tau_j)v_{j+1}. 
\end{align*}
We call the piece-wise linear trajectory $\big(\mathscr{Z}(t): \tau_0^- < t < \tau_n^+\big)$ mechanically \emph{$r$-consistent} or \emph{$r$-inconsistent}, if 
\begin{align}
\label{consist}
\min_{\tau_0 \le t \le \tau_n}
\min_{0\le j < n}
\abs{\mathscr{Z}(t) - \mathscr{Z}_j^\prime}=r, 
\qquad\text{ respectively, }\qquad
\min_{\tau_0 \le t \le \tau_n}
\min_{0\le j < n}
\abs{\mathscr{Z}(t) - \mathscr{Z}_j^\prime}<r.
\end{align}
Note, that by formal definition the minimum distance on the left hand side can not be strictly larger than $r$.  

Given two finite pieces of mechanically $r$-consistent trajectories $\big(\mathscr{Z}_{a}(t): \tau_{a,0}^- < t < \tau_{a,n_a}^+\big)$ and $\big(\mathscr{Z}_{b}(t): \tau_{b,0}^- < t < \tau_{b,n_b}^+\big)$, defined over non-overlapping time intervals: $[\tau_{a,0},\tau_{a,n_a}] \cap [\tau_{b,0},\tau_{b,n_b}]=\emptyset$, with $\tau_{a,n_a}\le\tau_{b,0}$,  we will call them mechanically \emph{$r$-compatible} or \emph{$r$-incompatible} if 
\begin{align}
\label{compat}
\begin{aligned}
\min
\{
\min_{\tau_{a,0} \le t \le \tau_{a,n_a}}
\min_{0< j \le n_b}
\abs{\mathscr{Z}_a(t) - \mathscr{Z}_{b,j}^\prime}, 
\min_{\tau_{b,0} \le t \le \tau_{b,n_b}}
\min_{0\le j < n_a}
\abs{\mathscr{Z}_b(t) - \mathscr{Z}_{a,j}^\prime}
\}
\ge r, 
\\
\min
\{
\min_{\tau_{a,0} \le t \le \tau_{a,n_a}}
\min_{0< j \le n_b}
\abs{\mathscr{Z}_a(t) - \mathscr{Z}_{b,j}^\prime}, 
\min_{\tau_{b,0} \le t \le \tau_{b,n_b}}
\min_{0\le j < n_a}
\abs{\mathscr{Z}_b(t) - \mathscr{Z}_{a,j}^\prime}
\}
< r, 
\end{aligned}
\end{align}
respectively. 

It is obvious that given a mechanically $r$-consistent trajectory, any non-overlapping parts of it are pairwise mechanically $r$-compatible, and given a finite number of non-overlapping mechanically $r$-consistent pieces of trajectories which are also pair-wise mechanically $r$-compatible their concatenation (in the most natural way) is mechanically $r$-consistent.

\subsection{An auxiliary process}
\label{ss: An auxiliary process}

It will be convenient to introduce a third, auxiliary process $t\mapsto Z(t)\in \R^{3}$, and consider the joint realisation of all three processes $t\mapsto (Y(t), X(t), Z(t))$ on the same probability space. This construction will not be needed until section \ref{s: Beyond Naive}, but this is the optimal logical point to introduce it. The reader may safely skip to section \ref{s: Naive} and come back here before turning to section  \ref{s: Beyond Naive}. 

The process $t\mapsto Z(t)$ will be a \emph{short-sighted} version of the true physical process $t\mapsto X(t)$ in the sense that in its construction only memory effects by the last seen scatterers are taken into account. That is: only \emph{direct recollisions} with the last seen scatterer and \emph{direct shadowings} by the last straight flight segment are incorporated, disregarding more complex memory effects. Later in the paper the following two basic facts will be shown: 

\begin{enumerate}[(a)]

\item
Up to times $T=T(r)=\ordo (r^{-2} \abs{\log r}^{-2})$ the trajectories of the short-sighted process $Z(t)$ and the true physical process $X(t)$ coincide. This is the main part of the proof, filling sections \ref{s: Proof of Proposition bw-legs} and \ref{s: Proof of Proposition Z=X in one leg} and concluded in Lemma \ref{lem: Theta is large}.

\item
The short-sighted process $Z(t)$ and the Markovian process $Y(t)$ stay sufficiently close together with probability tending to $1$ (as $r\to 0$). This is the content of Lemma \ref{lem: Z is close to Y}, its proof is relatively simple.

\end{enumerate}

Based on these two conclusions, the invariance principle \eqref{ipforY} can be transferred to the true physical process $X(t)$, thus yielding the invariance principle \eqref{ourip}.

Define the following indicator variables: 
\begin{align}
\label{eta}
\begin{aligned}
&
\wh\eta_{j}
=
\wh\eta(y_{j-2}, y_{j-1}, y_{j})
:=
\one
\left\{
\abs{y_{j-1}}<1 \text{ and }
\min_{0\le t\le \xi_{j-2}}
\abs{y_{j-1}+r \frac{u_{j-1}-u_{j}}{\abs{u_{j-1}-u_{j}}}+t u_{j-2}} < r
\right\},
\\[10pt]
&
\wt\eta_j
=
\wt\eta(y_{j-2}, y_{j-1}, y_{j})
:=
\one\left\{
\abs{y_{j-1}}<1 \text{ and } 
\min_{0\le t\le \xi_j}
\abs{y_{j-1}+r \frac{u_{j-1}-u_{j-2}}{\abs{u_{j-1}-u_{j-2}}}+t u_{j}} <r
\right\},
\\[10pt]
&\eta_j
:=
\max\{\wh\eta_j, \wt\eta_j\}.
\end{aligned}
\end{align}

Before constructing the auxiliary process $t\mapsto Z(t)$ we prove the following

\begin{lemma}
\label{lem:probab-corel-eta}
There exists a constant $C<\infty$ such that for any sequence of signatures $\ueps=(\epsilon_j)_{j\ge1}$ the following bounds hold
\begin{align}
\label{probab-eta}
\condexpect{\eta_j}{\ueps}
&
\le 
Cr, 
\\
\label{corel-eta}
\condexpect{\eta_j\eta_k}{\ueps}
&
\le 
\begin{cases}
C r^2 \abs{\log r} & \text{ if } \abs{j-k}=1, 
\\
C r^2  & \text{ if } \abs{j-k}>1.  
\end{cases}
\end{align}
\end{lemma}

\begin{proof}
[Proof of Lemma \ref{lem:probab-corel-eta}]
Define the following auxiliary, and simpler, indicators: 
\begin{align*}
\wh\eta^\prime_{j}
:=
\one
\left\{
 \angle(-u_{j-1}, u_{j-2})  < \frac{2r}{\xi_{j-1}}
\right\},
\qquad\qquad
\wt\eta^\prime_{j}
:=
\one
\left\{
\angle(-u_{j-1}, u_{j})  < \frac{2r}{\xi_{j-1}}
\right\}. 
\end{align*}
Here, and in the rest of the paper we use the notation 
\[
\angle: S^{2}\times S^{2} \to [0,\pi], 
\qquad\qquad
\angle(u,v):= \arccos (u\cdot v).
\]
Then, clearly, 
\begin{align*}
\wt\eta_j \le \wt\eta^\prime_j, 
\qquad
\wh\eta_j \le \wh\eta^\prime_j.  
\end{align*}
It is straightforward that the indicators $\left(\wh\eta^\prime_{j}: 1\le j<\infty\right)$, and likewise, the indicators 
\\
$\left(\wt\eta^\prime_{j}: 1\le j<\infty\right)$, are independent among themselves and one-dependent across the two sequences. This holds even if conditioned on the sequence of signatures $\ueps$. 

Therefore, the following simple computations prove the claim of the lemma. 
\begin{align*}
&
\condexpect{\wh\eta^\prime_j}{\ueps}
\le 
C r^2 \int_0^\infty e^{-y} \min\{y^{-2}, r^{-2}\} dy
\le 
Cr, 
\\
& 
\condexpect{\wt\eta^\prime_j}{\ueps}
\le 
C r^2 \int_0^\infty e^{-y} \min\{y^{-2}, r^{-2}\} dy
\le 
Cr, 
\\
&
\condexpect{\wh\eta^\prime_{j+1}\wt\eta^\prime_{j}}{\ueps}
\le 
C r^2 \int_0^\infty \int_0^\infty e^{-y}e^{-z} \min\{y^{-2},z^{-2}, r^{-2}\} dy dz
\le 
C r^2 \abs{\log r}. 
\end{align*}
We omit the elementary computational details.
\end{proof}

Lemma \ref{lem:probab-corel-eta} assures that, as $r\to0$,  with probability tending to $1$, up to time of order $T=T(r)=\ordo(r^{-2}\abs{\log r}^{-1})$ it will not occur that two neighbouring or next-neighbouring $\eta$-s happen to take the value $1$ which would obscure the following construction. 
 
The algorithmic definition of the process $t\mapsto Z(t)$, in terms of and adapted to the natural filtration of the flight process $t\mapsto Y(t)$, goes as follows. The process $t\mapsto Z(t)$ is constructed on the successive intervals $[\tau_{j-1}, \tau_j)$, $j=1,2, \dots$, as follows: 

\begin{enumerate}[$\circ$]

\item
({\tt No interference with the past.})
\\
If $\eta_j=0$ then for $\tau_{j-1}\le t\le \tau_j$, $Z(t)=Z(\tau_{j-1}) + \{t\} u_j$. 

\item
({\tt Direct shadowing.})
\\
If $\wh\eta_j=1$, then for $\tau_{j-1}\le t\le \tau_j$, $Z(t)=Z(\tau_{j-1}) + \{t\} u_{j-1}$. 

\item
({\tt Direct recollision with the last seen scatterer.})
\\
If $\wh\eta_j=0$ and $\wt\eta_j=1$ then, in the time interval  $\tau_{j-1}\le t\le \tau_j$  the trajectory $t\mapsto Z(t)$ is defined as that of a mechanical particle starting with initial position $Z(\tau_{j-1})$, initial velocity $\dot Z(\tau_{j-1}^+)=u_j$ and colliding elastically with two infinite-mass spherical scatterers of radius $r$ centred at the points 
\begin{align*}
Z(\tau_{j-1})+ r \frac{u_{j-1}-u_{j}}{\abs{u_{j-1}-u_{j}}},
\quad\text{ respectively }\quad 
Z(\tau_{j-2})-r \frac{u_{j-1}-u_{j-2}}{\abs{u_{j-1}-u_{j-2}}}. 
\end{align*}

\end{enumerate}

These steps define the process $t\mapsto Z(t)$ in a unique way and adapted to the natural filtration of the process $t\mapsto Y(t)$. As we already stressed, the basic facts about the process $t\mapsto Z(t)$ (listed earlier in this subsection) will be proved in later sections of the paper, and will have a key role in proving our main Theorem \ref{thm:tricky}.

Consistently with the notations adopted for the processes $Y(t)$ and $X(t)$, we denote $Z_k:= Z(\tau_k)$ for $k\ge0$.


\begin{figure}[ht!]
  \begin{center}    
    \includegraphics[width=0.9\textwidth]{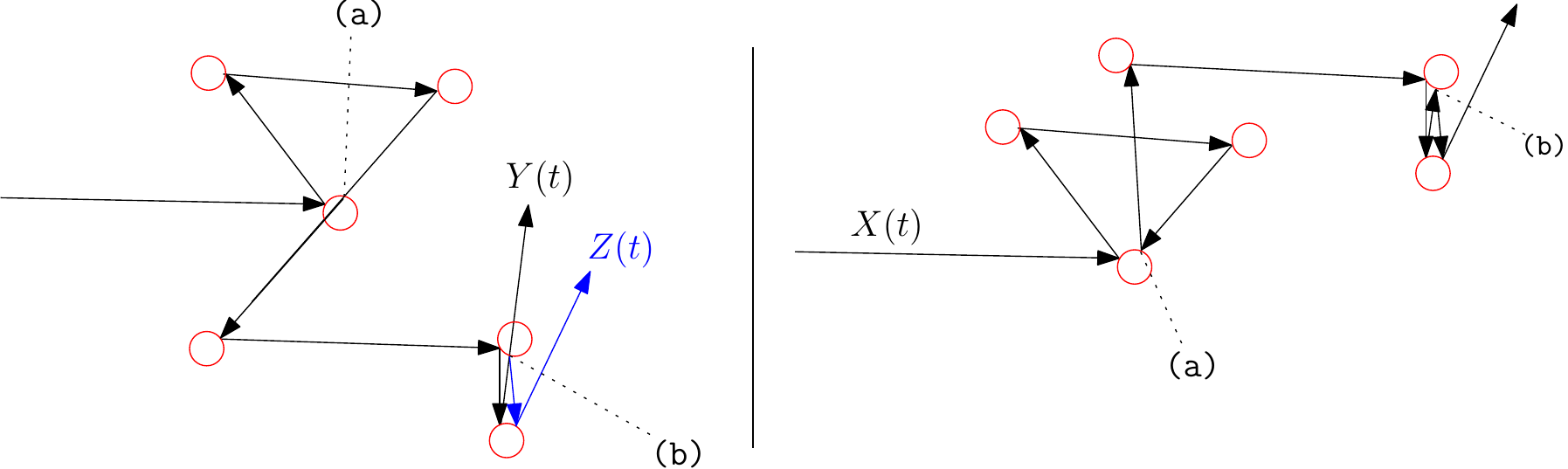}
  \end{center}
  \caption{%
    {\tt
    The above image shows a section of trajectory during which $X$, $Y$, and $Z$ would all three differ. On the left we see $Y$ and $Z$ remain together until point (b), where a direct recollision is respected by $Z$. Note that $Z$ ignores the mismatch at (a) as it is indirect. On the right, the process $X$ is coupled to $Y$ on the left. Note that $X$ respects the indirect recollision at point (a) and the direct recollision at (b).}
  }%
  
  \label{fig:X_Y_Z}

\end{figure}

\section{No mismatches up to $T=\ordo(r^{-1})$: Proof of Theorem \ref{thm: naive}} 
\label{s: Naive}

In this section we prove that the Markovian flight trajectory $Y(t)$, up to time scales of order $T=T(r)=\ordo(r^{-1})$, is mechanically $r$-consistent with probability $1-\ordo(1)$, and therefore the coupling bound of Theorem \ref{thm: naive} holds. On the way we establish various bounds to be used in later sections. This section is purely classical-probabilistic. It also prepares the ideas (and notation) for section \ref{s: Proof of Proposition bw-legs} where a similar argument is explored in more complex form. 

\subsection{Interferences}
\label{ss: Interferences}

Let $t \to Y(t)$ and $t \to Y^*(t)$ be two independent Markovian flight processes. Think about $Y(t)$ as running forward and $Y^*(t)$ as running backwards in time. (Note, that the Markovian flight process has invariant law under time reversal.)  Define the following events 
\begin{align*}
\wh W_j 
&:= 
\{\min \{ \abs{Y(t) - Y_j^{\prime}} : 0 < t < \tau_{j-1}\}<r\},
\\
\wt W_j 
&:= 
\{\min \{\abs{Y_k^{\prime}-Y(t)}: 0\le k< j-1, \ \ \ \tau_{j-1}< t < \tau_j \}<r\}, 
\\
\wh W_j^* 
&:= 
\{\min \{ \abs{Y^*(t) - Y_1^{\prime}} : 0 < t < \tau_{j-1}\}<r\},
\\
\wt W_j^{*} 
&:= 
\{\min \{\abs{Y^{*\prime}_k-Y(t)}: 0 < k \le j-1, \ \ \ 0 < t < \tau_1 \}<r\}, 
\\
\wh W_\infty^* 
&:= 
\{\min \{ \abs{Y^*(t) - Y^{\prime}_1} : 0 < t < \infty\}<r\},
\\
\wt W_\infty^* 
&:= 
\{\min \{\abs{Y^{*\prime}_k-Y(t)}: 0 <  k<\infty, \ \ \ 0< t < \tau_1 \}<r\}.
\end{align*}
In words $\wh W_j$ is the event that the virtual collision at $Y_j$ is \emph{shadowed} by the past path. While $\wt W_j$ is the event that in the time interval  $(\tau_{j-1},\tau_{j})$ there is a \emph{virtual recollision} with a past  scatterer.

It is obvious that 
\begin{align}
\label{hat-tilde-Ws}
\begin{aligned}
&
\probab{\wh W_j}
=
\probab{\wh W_j^*}
\le 
\probab{\wh W_{j+1}^*}
\le 
\probab{\wh W_{\infty}^*}, 
\\
&
\probab{\wt W_j}
=
\probab{\wt W_j^*}
\le 
\probab{\wt W_{j+1}^*}
\le 
\probab{\wt W_{\infty}^*}. 
\end{aligned}
\end{align}
On the other hand, by union bound and independence
\begin{align}
\label{pwhtbound-for-Y}
\begin{aligned}
\probab{\wh W^*_\infty}
&
\le 
\sum_{z\in \Z^{3}} 
\probab{ \{1 \le  k<\infty: Y^*_k\in B_{zr,2r}\} \not=\emptyset }
\probab{ \{0< t\le \xi_{1}: Y(t)\in B_{zr,2r}\} \not=\emptyset},
\\
&
\le 
\sum_{z\in \Z^{3}} 
(2r)^{-1} 
\expect{ \abs{ \{1<k<\infty: Y^*_k\in  B_{zr,2r}\} } } 
\expect{ \abs{ \{0< t\le \xi_{1}: Y(t)\in  B_{zr,3r}\} } },
\\
\probab{\wt W^*_\infty}
&
\le 
\sum_{z\in \Z^{3}} 
\probab{ \{0<t<\infty: Y^*(t)\in B_{zr,2r}\} \not=\emptyset }
\probab{ Y_1\in B_{zr,2r}},
\\
&
\le 
\sum_{z\in \Z^{3}} 
(2r)^{-1} 
\expect{ \abs{ \{0<t<\infty: Y^*(t)\in B_{zr,3r}\} } } 
\probab{ Y_1\in B_{zr,2r}},
\end{aligned}
\end{align}
recall that $\xi_1$ is the time of the first collision for the forwards process $t\mapsto Y(t)$. Here and in the rest of the paper we use the notation $\abs{\{\cdots\}}$ for either cardinality or Lebesgue measure of the set $\{\cdots\}$, depending on context.

\subsection{Occupation measures (Green's functions)}
\label{ss: Occupation-Naive}

Define the following occupation measures (Green's functions): for $A\subset\R^{3}$
\begin{align*}
  g(A) &:= \probab{ Y_1\in A } \\
  h(A) &:= \expect{ \abs{ \{0  < t \le \xi_1:  Y(t) \in A\} } }\\
  G(A) &:= \expect{ \abs{ \{1\le k  <  \infty: Y_k  \in A\} } }\\
  H(A) &:= \expect{ \abs{ \{0  < t  <  \infty: Y(t) \in A\} } }.
\end{align*}
Obviously, 
\begin{align}
\label{convolution-for-Y}
\begin{aligned}
  G(A) &= g(A) + \int_{\R^{3}} g(A-x)G(dx) \\
  H(A) &= h(A) + \int_{\R^{3}} h(A-x)G(dx) .
\end{aligned}
\end{align}

\subsection{Bounds}
\label{ss: Bounds - Naive}

\begin{lemma} 
\label{lem:greenbounds-for-Y}
  The following identities and upper bounds hold:
  \begin{align} 
  \label{gbound-for-Y}
    h(dx) = g(dx) &\le L(dx) \\
    \label{Gbound-for-Y}
    H(dx) = G(dx) &\le K(dx)+ L(dx),
      \end{align}
  where 

  \begin{align}
  \label{K and L}
    K(dx) := C\min\{1, \abs{x}^{-1}\} dx, 
    \qquad
    L(dx) &:= Ce^{-c\abs{x}} \abs{x}^{-2}dx,
\end{align}
with appropriately chosen $C<\infty$ and $c>0$.
\end{lemma}

\begin{proof}[Proof of Lemma \ref{lem:greenbounds-for-Y}]

The identity $h=g$ is a direct consequence of the flight length $\xi$ being $EXP(1)$-distributed. The distribution $g$ has the explicit expression
\begin{align*}
g(dx) 
= 
C \abs{x}^{-2}e^{-\abs{x}}dx
\end{align*}
from which the the upper bound \eqref{gbound-for-Y} follows. 

\eqref{Gbound-for-Y} then follows from \eqref{convolution-for-Y} and standard Green's function estimate for a random walk with step distribution $g$. 

\end{proof}

For later use we introduce the conditional versions -- conditioned on the sequence $\ueps$ (see \eqref{epsilon-ind}) -- of the bounds \eqref{gbound-for-Y}, \eqref{Gbound-for-Y}. In this order we define the conditional versions of the Green's functions, given $\eps\in\{0,1\}$, respectively $\ueps\in\{0,1\}^{\N}$:
\begin{align*}
  g_\epsilon (A) &:= \condprobab{ Y_1\in A }{\epsilon} \\
  h_\epsilon(A) &:= \condexpect{ \abs{ \{0  < t \le \xi_1:  Y(t) \in A\} } }{\epsilon}\\
  G_{\ueps}(A) &:= \condexpect{ \abs{ \{1\le k  <  \infty: Y_k  \in A\} } }{\ueps}\\
  H_{\ueps}(A) &:= \condexpect{ \abs{ \{0  < t  <  \infty: Y(t) \in A\} } }{\ueps},
\end{align*}
and state the conditional version of Lemma \ref{lem:greenbounds-for-Y}:

\begin{lemma}
\label{lem:greenbounds-for-Y-conditional}
The following upper bounds hold uniformly in $\ueps\in\{0,1\}^{\N}$:
\begin{align} 
\label{ghbound-for-Y-conditional}
&
g_\epsilon (dx) \le L(dx), 
&&
h_\epsilon (dx) \le L(dx), 
\\
\label{GHbound-for-Y-conditional}
&
G_{\ueps}(dx) \le K(dx)+ L(dx), 
&&
H_{\ueps}(dx) \le K(dx)+ L(dx), 
\end{align}
with $K(x)$ and $L(x)$ as in \eqref{K and L}, with appropriately chosen constants $C<\infty$ and $c>0$. 
\end{lemma}

\begin{proof}
[Proof of Lemma \ref{lem:greenbounds-for-Y-conditional}]
Noting that 
\begin{align*}
g_\epsilon(dx) 
\le
C \abs{x}^{-2}e^{-\abs{x}}dx, 
\qquad
h_\epsilon(dx) 
\le
C \abs{x}^{-2}e^{-\abs{x}}dx, 
\end{align*}
the proof of Lemma \ref{lem:greenbounds-for-Y-conditional} follows very much the same lines as the proof of Lemma \ref{lem:greenbounds-for-Y}. We omit the details. 
\end{proof}

\subsection{Computation} 
\label{ss: Computation -- Naive}

According to \eqref{hat-tilde-Ws}, \eqref{pwhtbound-for-Y}, for every $j =1,2,\dots$
\begin{align*}
\begin{aligned}  
  \probab{\wh W_j} 
  &\le 
  \probab{\wh W_\infty^*}
  \le 
  (2r)^{-1}\sum_{z\in\Z^{3}}G(B_{zr, 2r})h(B_{zr, 3r}), 
  \\
  \probab{\wt W_j} 
  &\le 
  \probab{\wt W^*_\infty}
  \le 
  (2r)^{-1}\sum_{z\in\Z^{3}}H(B_{zr, 3r})g(B_{zr, 2r}).
\end{aligned}
\end{align*}

Moreover, straightforward computations yield 

\begin{lemma} \label{lem:KL bounds}
In dimension $d= 3$ the following bounds hold, with some $C<\infty$
\begin{align}
\label{KL-LL}
\sum_{z \in \Z^{3}} K(B_{zr, 3r}) L(B_{zr, 2r})\le  C r^{3}, 
\qquad\qquad
\sum_{z \in \Z^{3}} L(B_{zr, 3r}) L(B_{zr, 2r})\le  C r^2.
\end{align}
\end{lemma}

\begin{proof}[Proof of Lemma \ref{lem:KL bounds}]
The bounds \eqref{KL-LL} readily follow from explicit computations. We omit the details.  

\end{proof}

We conclude this section with the following consequence of the above arguments and computations.

\begin{corollary}
\label{cor: probaWbounds-forY}
There exists a constant $C<\infty$ such that for any $j\ge 1$: 
  \begin{align}
  \label{probaWbounds-forY}
    \probab{\wt W_j} \le Cr, 
    \qquad
    \probab{\wh W_j}  \le Cr.
  \end{align}
\end{corollary}

\subsection{No mismatching -- up to $T\sim \ordo(r^{-1})$}
\label{ss: No interference-Naive}

Define the stopping time
\begin{align*}  
\sigma
:= 
\min \{j>0: \max\{\one_{\wh W_j}, \one_{\wt W_j}\} =1\}, 
\end{align*}
and note that by construction
\begin{align}
\label{no-mismatch-till-Theta}
\inf\{t>0: X(t)\not=Y(t)\} \ge \tau_{\sigma-1}. 
\end{align}

\begin{lemma}
\label{lem:Theta-is-large-for-Y}
Let $T=T(r)$ be such that $\lim_{r\to0}T(r)=\infty$ and $\lim_{r\to0} rT(r)=0$. Then 
\begin{align}
\label{Theta-is-large-for-Y}
\lim_{r\to0}
\probab{\tau_{\sigma-1}<T} =0. 
\end{align}
\end{lemma}

\begin{proof}
[Proof of Lemma \ref{lem:Theta-is-large-for-Y}]
\begin{align}
\label{trivi-0}
\probab{\tau_{\sigma-1}<T}
\le 
\probab{\sigma\le 2 T}
+
\probab{\sum_{j=1}^{2 T-1} \xi_j<T}
\le
Cr T
+
C e^{-c T}, 
\end{align}
where $C<\infty$ and $c>0$. The first term in the middle expression of \eqref{trivi-0} is bounded  by union bound and \eqref{probaWbounds-forY} of Corollary \ref{cor: probaWbounds-forY}. In bounding the second term we use a large deviation upper bound for the sum of independent $EXP(1)$-distributed  $\xi_j$-s. 

Finally, \eqref{Theta-is-large-for-Y} readily follows from \eqref{trivi-0}. 
\end{proof}

\eqref{naive} follows directly from \eqref{no-mismatch-till-Theta} and \eqref{Theta-is-large-for-Y}, and this concludes the proof of Theorem \ref{thm: naive}. \qed

\section{Beyond the na\"{i}ve coupling}
\label{s: Beyond Naive}

The forthcoming parts of the paper rely on the joint realisation (coupling) of the \emph{three} processes $t\mapsto\big(Y(t), X(t), Z(t)\big)$ as described in section \ref{s: Construction}. In particular, recall the construction of the process $t\mapsto Z(t)$ from section \ref{ss: An auxiliary process}.

\subsection{Breaking $Z$ into legs}
\label{ss: Breaking $Z$ into legs}

Let $\Gamma_{0}:=0$, $\Theta_0=0$ and for $n\ge1$
\begin{align}
\label{Gamma-Theta}
\begin{aligned}
&
\Gamma_n
:=
\min\{j\ge \Gamma_{n-1}+2: \min\{\xi_{j-1}, \xi_j, \xi_{j+1}, \xi_{j+2}\}>1\}, 
\qquad\qquad
&&
\gamma_n
:=
\Gamma_n-\Gamma_{n-1}, 
\\
&
\Theta_n
:=
\tau_{\Gamma_n},
&&
\theta_n
:=
\Theta_n-\Theta_{n-1},
\end{aligned}
\end{align}
and denote 
\begin{align*}
&
\xi_{n,j}
:=
\xi_{\Gamma_{n-1}+j}, 
\quad
u_{n,j}
:=
u_{\Gamma_{n-1}+j}, 
\quad
y_{n,j}
:=
y_{\Gamma_{n-1}+j}, 
&&
1\le j\le \gamma_n,
\\
&
Y_n(t)
:=
Y(\Theta_{n-1}+t) -Y(\Theta_{n-1}), 
&&
0\le t\le \theta_n, 
\\
&
Z_n(t)
:=
Z(\Theta_{n-1}+t) -Z(\Theta_{n-1}), 
&&
0\le t\le \theta_n. 
\end{align*}
Then, it is straightforward that the \emph{packs of random variables} 
\begin{align}
\label{packs}
\varpi_n:=
\left(\gamma_n; (\xi_{n,j}, u_{n,j}): 1\le j\le \gamma_n\right), 
\qquad n \ge 1, 
\end{align}
are fully independent (for $n \ge 1$), and also identically distributed for $n \ge 2$. (The first pack is deficient if $\min\{\xi_0,\xi_1\}<1$.) It is also straightforward that the \emph{legs} of the Markovian flight process
\begin{align*}
\left(\theta_n; Y_n(t): 0\le t \le \theta_n \right), 
\qquad
n \ge 1,
\end{align*}
are fully independent, and identically distributed for $n \ge 2$. 

A \emph{key observation is} that due to the rules of construction of the process $t\mapsto Z(t)$ exposed in section \ref{ss: An auxiliary process}, the legs 
\begin{align}
\label{Z-legs}
\left(\theta_n; Z_n(t): 0\le t \le \theta_n\right), 
\qquad
n \ge 1,
\end{align}
of the auxiliary process $t\mapsto Z(t)$ \emph{are also independently constructed} from the packs \eqref{packs}, following the rules in section \ref{ss: An auxiliary process}. Note, that the restrictions $\abs{y_{j-1}}<1$ in \eqref{eta} were imposed exactly in order to ensure this independence of the legs \eqref{Z-legs}. Therefore we will construct now the auxiliary process $t\mapsto Z(t)$ and its time reversal  $t\mapsto Z^*(t)$ from an infinite sequence of independent packs \eqref{packs}. In order to reduce unnecessary complications of notation from now on we assume $\min\{\xi_0,\xi_1\}>1$. 

\medskip
\noindent
{\bf Remark:}
In order to break up the auxiliary process $t\mapsto Z(t)$ into \emph{independent legs} the choice of simpler stopping times
\begin{align*}
\Gamma^{\prime}_n
:=
\min\{j\ge \Gamma_{n-1}+1: \min\{\xi_j, \xi_{j+1}\}>1\}, 
\end{align*}
would work. However, we need the slightly more complicated stoppings $\Gamma_n$, given in \eqref{Gamma-Theta}, for some other reasons which will become clear towards the end of section \ref{ss: One leg} and in the statement and proof of Lemma \ref{lem:greenbounds-for-Z}. 

\subsection{One leg}
\label{ss: One leg}

Let $\xi_j$, $u_j$, $j\ge1$, be fully independent random variables with the distributions \eqref{seqingred}, conditioned to 
\[
\min\{\xi_1,\xi_2\}>1,
\] 
and $y_j$ as in \eqref{ysteps}. Let 
\begin{align}
\label{gamma-stop}
\gamma:=\min\{j\ge 2: 
\min\{\xi_{j-1}, \xi_{j}, \xi_{j+1}, \xi_{j+2} \}>1\} \in \{2\} \cup \{5, 6,  \dots\}.
\end{align}
Note that $\gamma$ can not assume the values $\{1,3,4\}$.
Call 
\begin{align}
\label{packdef}
\varpi
:=
\left(\gamma; (\xi_j,  u_j) : 1\le j \le \gamma\right)
\end{align}
a \emph{pack}, and keep the notation $\tau_j:=\sum_{k=1}^j \xi_k$, and $\theta:=\tau_\gamma$. 

The \emph{forward leg} 
\begin{align*}
\left(\theta; Z(t): 0\le t \le \theta\right)
\end{align*}
is constructed from the pack $\varpi$ according to the rules given in section \ref{ss: An auxiliary process}. We will also denote 
\begin{align*}
Z_j:= Z(\tau_j), \quad 0\le j\le \gamma;
\qquad\qquad
\overline{Z}:=  Z_\gamma = Z(\theta). 
\end{align*}
These are the discrete steps, respectively, the terminal position of the leg. 

It is easy to see that the distributions of $\gamma$ and $\theta$ are exponentially tight: there exist constants $C<\infty$ and $c>0$ such that for any $s\in[0,\infty)$
\begin{align}
\label{gamma-theta-expo-tight}
\probab{\gamma>s}
\le Ce^{-c s}, 
\qquad\qquad
\probab{\theta>s}
\le Ce^{-c s}.
\end{align}

The \emph{backwards leg}
\begin{align*}
\left(\theta; Z^*(t): 0\le t \le \theta\right)
\end{align*}
is constructed from the pack $\varpi$  as 
\begin{align*}
Z^*(t, \varpi)
:=
Z(\theta-t, \varpi^*)
-
\overline{Z}(\varpi^*),
\end{align*}
where the backwards pack
\begin{align*}
\varpi^*
:=
\left(\gamma; (\xi_{\gamma-j}, -u_{\gamma-j}): 0\le j \le \gamma-1\right)
\end{align*}
is the time reversion of the pack $\varpi$.

Note that since the velocities $u_n$ are uniformly distributed the forward and backward packs, $\varpi$ and $\varpi^*$, are identically distributed. However, since under time reversal, recollisions become shadowed scattering and vice-versa, the forward and backward processes $\big(t\mapsto  Z(t): 0\le t\le \theta \big)$ and  $\big(t\mapsto Z^*(t): 0\le t\le \theta \big)$ are not identically distributed. 
The backwards process $t\mapsto Z^*(t)$ could also be defined using an explicit step-by-step construction, similar (but not identical) to those in section \ref{ss: An auxiliary process}, but we will not rely on these step-wise rules and therefore omit their explicit formulation. Thus it suffices to take $\big(t\mapsto Z^*(t): 0\le t\le \theta \big)$ to be the time-reversal of $\big(t\mapsto Z(t): 0\le t\le \theta \big)$.

Consistent with the previous notation, we denote
\begin{align*}
Z^*_j:= Z^*(\tau_j), \quad 0\le j\le \gamma;
\qquad
\overline {Z}^*:=  Z^*_\gamma = Z^*(\theta)=-\overline{Z}. 
\end{align*}
Note, that due to the construction rules of the forward and backward legs, their beginning, middle and ending parts
\begin{align}
\label{broken-forward-legs}
\begin{aligned}
&
\left(\tau_1; Z(t): 0\le t \le \tau_1\right), 
\\
&
\left(\tau_{\gamma-1}-\tau_1; Z(\tau_1+t)-Z(\tau_1): 0\le t \le \tau_{\gamma-1}-\tau_1\right),
\\
&
\left(\tau_\gamma-\tau_{\gamma-1}; Z(\tau_{\gamma-1}+t)-Z(\tau_{\gamma-1}): 0\le t \le \tau_\gamma-\tau_{\gamma-1}\right), 
\end{aligned}
\end{align}
are \emph{independent}, and likewise for the backwards process $Z^*$, \begin{align}
\label{broken-backward-legs}
\begin{aligned}
&
\left(\tau_1; Z^*(t): 0\le t \le \tau_1\right), 
\\
&
\left(\tau_{\gamma-1}-\tau_1; Z^*(\tau_1+t)-Z^*(\tau_1): 0\le t \le \tau_{\gamma-1}-\tau_1\right),
\\
&
\left(\tau_\gamma-\tau_{\gamma-1}; Z^*(\tau_{\gamma-1}+t)-Z^*(\tau_{\gamma-1}): 0\le t \le \tau_\gamma-\tau_{\gamma-1}\right).
\end{aligned}
\end{align}
This fact will be of crucial importance in the proof of Proposition \ref{prop:no-interference-between-legs}, section \ref{ss: Bounds for Z} below. This is the reason (alluded to in the remark at the end of section \ref{ss: Breaking $Z$ into legs}) we chose the somewhat complicated stopping time as defined in \eqref{gamma-stop}.

\subsection{Multi-leg concatenation}
\label{ss: Multi-leg concatenation}

Let $\varpi_n = \left(\gamma_n; \;  (\xi_{n,j},u_{n,j}): 1\le j\le \gamma_n\right)$, $n\ge 1$,  be a sequence of i.i.d \emph{packs} \eqref{packdef}, and denote $\theta_n$, $(Z_n(t): 0\le t\le \theta_n)$, $(Z_{n,j}: 1\le j\le\gamma_n)$,  $(Z^*_n(t): 0\le t\le \theta_n)$, $(Z^*_{n,j}: 1\le j\le\gamma_n)$, $\overline{Z}_n$, $\overline{Z}^*_n$ the various objects defined in section \ref{ss: One leg}, specified for the $n$-th independent leg.     

In order to construct the concatenated forward and backward  processes $t\mapsto Z(t)$, $t\mapsto Z^*(t)$, $0\le t<\infty$, we first define for $n\in\Z_+$, respectively $t\in\R_+$
\begin{align*}
&
\Gamma_n 
:=  
\sum_{k=1}^n \gamma_k,      
&&  
\nu_n
:=  
\max\{m:\Gamma_m\le n\},       
&&
\{n\}
:=
n-\Gamma_{\nu_n}, 
\\
&  
\Theta_n
:=
\sum_{k=1}^n \theta_k,        
&&
\nu_t
:=
\max\{m:\Theta_m<t\},        
&&  
\{t\}
:=
t-\Theta_{\nu_t}. 
\end{align*}

The concatenated (multi-leg) forward and backward $Z$-processes are 
\begin{align}
\label{Xi-walk}
\begin{aligned}
&
\Xi_n
:= 
\sum_{k=1}^n
\overline{Z}_k,
\qquad
&&
\qquad
Z_n
:=
\Xi_{\nu_n} + Z_{\nu_n+1, \{n\}},
\qquad
&&
\qquad
Z(t)
:=
\Xi_{\nu_t} + Z_{\nu_t+1}(\{t\}), 
\\
&
\Xi^*_n
:= 
\sum_{k=1}^n
\overline{Z}^*_k,
\qquad
&&
\qquad
Z^*_n
:=
\Xi^*_{\nu_n} + Z^*_{\nu_n+1, \{n\}},
\qquad
&&
\qquad
Z^*(t)
:=
\Xi^*_{\nu_t} + Z^*_{\nu_t+1}(\{t\}).
\end{aligned}
\end{align}
Note that $\Xi_n$ and $\Xi^*_n$ are random walks with independent steps; $t\mapsto Z(t)$, $0\le t<\infty$,  is exactly the $Z$-process constructed in section \ref{ss: An auxiliary process},  with $Z_n=Z(\tau_n)$, $0\le n<\infty$. Similarly, $t\mapsto Z^*(t)$, $0\le t<\infty$,  is the time reversal of the $Z$-process and $Z^*_n=Z^*(\tau_n)$, $0\le n<\infty$.

Theorem \ref{thm:tricky} will follow from Propositions \ref{prop: Z=X in one leg} and \ref{prop:no-interference-between-legs} of the next two sections. 

\subsection{Mismatches within one leg}
\label{ss: No mismatch within one leg}

Given a pack $\varpi=\left(\gamma; (\xi_j, u_j): 1\le j\le\gamma\right)$ \eqref{packdef}, and \emph{arbitrary} incoming and outgoing velocities $u_0, u_{\gamma+1}\in S^{2}$ let $\big((Y(t), \cX(t), Z(t)): 0^-< t< \theta^+\big)$, be the triplet of Markovian flight process, Lorentz exploration process and auxiliary $Z$-process jointly constructed with these data, as described in sections \ref{ss: Ingredients and the Markovian flight process}, \ref{ss: The Lorentz exploration process}, respectively, \ref{ss: An auxiliary process}. We use the notation $\cX$ to denote a mechanical Lorentz exploration constructed using the rules of section \ref{ss: The Lorentz exploration process} defined on a leg for times $t\in [0^-,\theta^+]$, and independently for different legs. By $0^-<t<\theta^+$ we mean that the incoming velocities at $0^-$ are given as $\dot Y(0^-)=\dot\cX(0^-)=\dot Z(0^-)=u_0$ and the outgoing velocities at $\theta^+$ are $\dot Y(\theta^+)=\dot Z(\theta^+)=u_{\gamma+1}$, while $\dot\cX(\theta^+)$ is determined by the construction from section \ref{ss: The Lorentz exploration process}. That is,  $\dot\cX(\theta^+)=u_{\gamma+1}$ if this last scattering is not shadowed by the trajectory $\big( \cX(t): 0\le t\le \theta\big)$ and $\dot\cX(\theta^+)=\dot\cX(\theta^-)$ if it is shadowed.

\begin{proposition}
  \label{prop: Z=X in one leg}
  There exists a constant $C<\infty$ such that for any $u_0, u_{\gamma+1}\in S^{2}$
  \begin{align}  \label{Z neq X in one leg}
    \probab{\cX(t)\not\equiv Z(t):0^-< t< \theta^+}
    \le C r^{2}\abs{\log r}^2.
  \end{align} 
\end{proposition} 

The proof of this Proposition relies on controlling the geometry of mismatchings, and is postponed until Section \ref{s: Proof of Proposition Z=X in one leg}.

\subsection{Inter-leg mismatches} 
\label{ss: Inter-leg mismatches}

Let $t \to Z(t)$ be a forward $Z$-process built up as concatenation of legs, as exposed in section \ref{ss: Multi-leg concatenation} and  define the following events 
\begin{align}
\label{What-Wtilde}
\begin{aligned}
\wh W_j 
&:= 
\big\{
\min \{ \abs{Z(t) - Z^{\prime}_k} : 
&&
0 < t < \Theta_{j-1}, 
&&
\Gamma_{j-1}< k \le \Gamma_j\}<r
\big\},
\\
\wt W_j 
&:= 
\big\{
\min \{\abs{Z^{\prime}_k-Z(t)}: 
&&
0\le k< \Gamma_{j-1},
&&
\Theta_{j-1}< t < \Theta_j \}<r
\big\}.
\end{aligned}
\end{align}
In words $\wh W_j$ is the event that a collision occuring in the $j$-th leg is \emph{shadowed} by the past path. While $\wt W_j$ is the event that within the $j$-th leg the $Z$-trajectory bumps into a scatterer placed in an earlier leg. That is, $\wt W_j \cup \wh W_j$ is precisely the event that the concatenated first $j-1$ legs and the $j$-th leg are mechanically $r$-incompatible (see section \ref{ss: Mechanical consistency and compatibility}). 

The following proposition indicates that on our time scales there are no ``inter-leg mismatches'':

\begin{proposition}
\label{prop:no-interference-between-legs}
There exists a constant $C<\infty$ such that for all $j\ge 1$
\begin{align}
\label{no-interference-between-legs}
\probab{\wt W_j} 
\le    
Cr^2, 
\qquad
\probab{\wh W_j} 
\le    
Cr^2.
\end{align}
\end{proposition}
The proof of Proposition \ref{prop:no-interference-between-legs} is the content of section \ref{s: Proof of Proposition bw-legs}.

\section{Proof of Proposition \ref{prop:no-interference-between-legs}} 
\label{s: Proof of Proposition bw-legs}

This section is purely probabilistic and of similar spirit as section \ref{s: Naive}. The notation used is also similar. However, similar is not identical. The various Green's functions used here, although denoted $g,h,G,H$, as in section \ref{s: Naive}, are similar in their r\^ole but not the same. The estimates on them are also different. 

\subsection{Occupation measures (Green's functions)}
\label{ss: Occupation measures}

Let now $t\mapsto Z^*(t)$, $0\le t<\infty$, be a backward $Z^*$-process and $t\mapsto Z(t)$, $0\le t\le \theta$,  a forward one-leg $Z$-process, assumed independent. In analogy with the events $\wh W_j$ and $\wt W_j$ defined in \eqref{What-Wtilde} we define  
\begin{align*}
\wh W_j^* 
&:= 
\big\{
\min \{ \abs{Z^*(t) - Z^{\prime}_k} : 
&&
0 < t < \Theta_{j-1},
&&
0 < k \le \gamma\}<r
\big\},
\\
\wt W_j^* 
&:= 
\big\{
\min \{\abs{Z^{*\prime}_k-Z(t)}: 
&&
0 < k \le \Gamma_{j-1},
&&
0 < t < \theta \}<r
\big\},
\\
\wh W_\infty^* 
&:= 
\big\{
\min \{ \abs{Z^*(t) - Z^{\prime}_k} : 
&&
0 < t < \infty,
&&
0 < k \le \gamma\}<r
\big\},
\\
\wt W_\infty^* 
&:= 
\big\{
\min \{\abs{Z^{*\prime}_k-Z(t)}: 
&&
0 < k < \infty,
&&
0 < t < \theta \}<r
\big\}.
\end{align*}

It is obvious that 
\begin{align}
\label{hatWs-tildeWs-for-Z}
\begin{aligned}
&
\probab{\wh W_j}
=
\probab{\wh W_j^*}
\le 
\probab{\wh W_{j+1}^*}
\le 
\probab{\wh W_{\infty}^*}, 
\\
&
\probab{\wt W_j}
=
\probab{\wt W_j^*}
\le 
\probab{\wt W_{j+1}^*}
\le 
\probab{\wt W_{\infty}^*}. 
\end{aligned}
\end{align}
On the other hand, by the union bound and independence we have 
\begin{align}
\label{pwhtbound-for-Z}
\begin{aligned}
\probab{\wh W^*_\infty}
&
\le 
\sum_{z\in \Z^{3}} 
\probab{ \{0<t<\infty: Z^*(t)\in B_{zr,2r}\} \not=\emptyset }
\probab{ \{1\le k\le \gamma: Z_k\in B_{zr,2r}\} \not=\emptyset}
\\
&
\le 
\sum_{z\in \Z^{3}} 
(2r)^{-1} 
\expect{ \abs{ \{0<t<\infty: Z^*(t)\in B_{zr,3r}\} } } 
\expect{ \abs{ \{1\le k\le \gamma: Z_k\in B_{zr,2r}\} } }
\\
\probab{\wt W^*_\infty}
&
\le 
\sum_{z\in \Z^{3}} 
\probab{ \{1\le k<\infty: Z^*_k\in B_{zr,2r}\} \not=\emptyset }
\probab{ \{0< t\le \theta: Z(t)\in B_{zr,2r}\} \not=\emptyset},
\\
&
\le 
\sum_{z\in \Z^{3}} 
(2r)^{-1} 
\expect{ \abs{ \{1 \le k<\infty: Z^*_k\in  B_{zr,2r}\} } } 
\expect{ \abs{ \{0< t\le \theta: Z(t)\in  B_{zr,3r}\} } }.
\end{aligned}
\end{align}
Therefore, in view of \eqref{hatWs-tildeWs-for-Z} we have to control the mean occupation time measures appearing on the right hand side of \eqref{pwhtbound-for-Z}.

Define the following mean occupation measures (Green's functions): for $A\subset\R^{3}$ let
\begin{align*}
g(A)    
&
:=  
\expect{ \abs{ \{1\le k\le \gamma: Z_k\in A\} } },
\\
g^*(A)  
&
:=  
\expect{ \abs{ \{1\le k\le \gamma: Z^*_k\in A\} } },  
\\ 
h(A)    
&
:=  
\expect{ \abs{ \{0< t \le \theta: Z(t)\in A\} } },    
\\
h^*(A)  
&
:=  
\expect{ \abs{ \{0< t \le \theta: Z^*(t)\in A\} } },   
\\
R^*(A)  
&
:=  
\expect{ \abs{ \{1\le n< \infty: \Xi^*_n\in A\} } },  
\\
G^*(A)  
&
:=  
\expect{ \abs{ \{1\le k< \infty: Z^*_k\in A\} } },    
\\
H^*(A)  
&
:=  
\expect{ \abs{ \{0< t< \infty: Z^*(t)\in A\} } }.
\end{align*}
It is obvious that 
\begin{align}
\label{convolution-1}
\begin{aligned}
G^*(A) 
&= 
g^*(A) + \int_{\R^{3}} g^*(A-x)R^*(dx), 
\\
H^*(A) 
&= 
h^*(A) + \int_{\R^{3}} h^*(A-x)R^*(dx).
\end{aligned}
\end{align}

\subsection{Bounds}
\label{ss: Bounds for Z}

\begin{lemma}
\label{lem:greenbounds-for-Z}
The following upper bounds hold:
\begin{align}
\label{ghbound-for-Z}
&
\max\{g (dx),g^{*} (dx)\}\le  M(dx), 
&&
\max\{h (dx),h^{*} (dx)\}\le  L(dx),    
\\[8pt]
\label{R*bound-for-Z}
&
R^* (dx)\le  K(dx),  
\\[8pt]
\label{GH*bound-for-Z}
&
G^* (dx)\le  K(dx), 
&&
H^* (dx)\le  K(dx) + L(dx),
\end{align}
where
\begin{align*}
K(dx) := C\min\{1, \abs{x}^{-1}\} dx, 
\qquad\qquad
L(dx) := Ce^{-c\abs{x}} \abs{x}^{-2}dx,
\qquad\qquad
M(dx) &:= Ce^{-c\abs{x}} dx, 
\end{align*}
with appropriately chosen $C<\infty$ and $c>0$.
\end{lemma}

\begin{proof}
[Proof of Lemma \ref{lem:greenbounds-for-Z}]

The proof of the bounds \eqref{ghbound-for-Z} hinges on the decompositions \eqref{broken-forward-legs} and \eqref{broken-backward-legs} of the forward and backward legs into independent parts.  

Let 
\begin{align}
\label{g-and-h}
\begin{aligned}
g_1(A)
&
:=
\probab{Z_1\in A}
=
\probab{Z^*_1\in A}
&&
=
C \int_A \one (\abs{x}>1)e^{-\abs{x}}dx,
\\
h_1(A)
&
:=
\expect{\abs{\{t\le\tau_1: Z(t)\in A\}}}
=
\expect{\abs{\{t\le\tau_1: Z^*(t)\in A\}}}
&&
=
C^{\prime} \int_A \abs{x}^{-2} e^{-\max\{1,\abs{x}\}} dx, 
\end{aligned}
\end{align}
and 
\begin{align*}
g_{2}(A)    
&
:=  
\expect{ \abs{ \{1\le k\le \gamma: Z_k-Z_1\in A\} } },   
\\
g^{*}_{2}(A)  
&
:=  
\expect{ \abs{ \{1\le k\le \gamma: Z^*_k-Z^*_1\in A\} } },  
\\ 
h_{2}(A)    
&
:=  
\expect{ \abs{ \{0< t \le \theta-\tau_1: Z(\tau_1+t)-Z_1 \in A\} } },     
\\
h^{*}_{2}(A)  
&
:=  
\expect{ \abs{ \{0< t \le \theta-\tau_1: Z^*(\tau_1+t)-Z^*_1 \in A\} } }.     
\end{align*}
Due to the exponential tail of the distribution of $\gamma$ and $\theta$, \eqref{gamma-theta-expo-tight}, there are constants $C<\infty$ and $c>0$ such that for any $s<\infty$
\begin{align}
\label{expo-decay}
\begin{aligned}
&
\max\{ 
g_{2}(\{x: \abs{x}>s\}), 
g^{*}_{2}(\{x: \abs{x}>s\})\}
\le C e^{-c s}, 
\\
&
\max\{ 
h_{2}(\{x: \abs{x}>s\}), 
h^{*}_{2}(\{x: \abs{x}>s\})
\}
\le C e^{-c s},
\end{aligned}
\end{align}
and furthermore, 
\begin{align}
\label{finite}
\begin{aligned}
&
g_{2}(\R^{3})=
g^{*}_{2}(\R^{3})=
\expect{\gamma}<\infty, 
\\
&
h_{2}(\R^{3})=
h^{*}_{2}(\R^{3})=
\expect{\theta-\tau_1}<\infty.  
\end{aligned}
\end{align}
From the independent decompositions \eqref{broken-backward-legs} and \eqref{broken-forward-legs} it follows that 
\begin{align}
\label{convolutions-2}
\begin{aligned}
&
g(A)
=
\int_{\R^{3}} g_{2}(A-x) g_1(dx),
&&
g^*(A)
=
\int_{\R^{3}} g^{*}_{2}(A-x) g_1(dx), 
\\
&
h(A)
=
\int_{\R^{3}} h_{2}(A-x) g_1(dx)
+
h_1(A),
&&
h^*(A)
=
\int_{\R^{3}} h^{*}_{2}(A-x) g_1(dx)
+
h_1(A). 
\end{aligned}
\end{align}
The bounds \eqref{ghbound-for-Z} readily follow from the explicit expressions \eqref{g-and-h}, the convolutions \eqref{convolutions-2} and the bounds \eqref{expo-decay} and \eqref{finite}.  

The bound \eqref{R*bound-for-Z} is a straightforward Green's function bound for the the random walk $\Xi^*_n$ defined in \eqref{Xi-walk}, by noting that the distribution of the i.i.d. steps $\overline{Z}^*_k$ of this random walk has bounded density and exponential tail decay. 

Finally, the bounds \eqref{GH*bound-for-Z} follow from the convolutions \eqref{convolution-1} and the bounds \eqref{ghbound-for-Z}, \eqref{R*bound-for-Z}. 

\end{proof}

\noindent
{\bf Remark:} 
On the difference between Lemmas \ref{lem:greenbounds-for-Y} and \ref{lem:greenbounds-for-Z}. Note the difference between the upper bounds for $g$ in \eqref{gbound-for-Y}, respectively,   \eqref{ghbound-for-Z}, and on $G$ in \eqref{Gbound-for-Y}, respectively, \eqref{GH*bound-for-Z}. These are important and are due to the fact that the length first step in a $Z$- or $Z^*$-leg is distributed as $(\xi \,|\, \xi>1) \sim EXP(1|0)$ rather than $\xi \sim EXP(1)$. 

\subsection{Computation}
\label{ss: Computation for Z}

According to \eqref{pwhtbound-for-Z}
\begin{align}
\label{pwthwbound2-for-Z}
\begin{aligned}
&
\probab{\wt W_j}
\le 
\probab{\wt W^*_\infty}
\le
(2r)^{-1} \sum_{z\in\Z^{3}} H^*(B_{zr, 3r})g(B_{zr, 2r}),  
\\
&
\probab{\wh W_j}
\le 
\probab{\wh W^*_\infty}
\le 
(2r)^{-1} \sum_{z\in\Z^{3}} G^*(B_{zr, 2r})h(B_{zr, 3r}).
\end{aligned}
\end{align}

\begin{lemma}
\label{lem:KLM bounds}
In dimension $d= 3$ the following bounds hold, with some $C<\infty$
\begin{align}
\label{KM-LM}
&
\sum_{z\in\Z^{3}} K(B_{zr, 3r})M(B_{zr, 2r})\le  C r^{3}, 
&&
\sum_{z\in\Z^{3}} M(B_{zr, 3r})L(B_{zr, 2r})\le  C r^{3}.
\end{align}
\end{lemma}

\begin{proof}
[Proof of Lemma \ref{lem:KLM bounds}]
The bounds \eqref{KM-LM} (similarly to the bounds \eqref{KL-LL}) readily follow from explicit computations which we omit. 
\end{proof}

\begin{proof}
[Proof of Proposition \ref{prop:no-interference-between-legs}]
Proposition \ref{prop:no-interference-between-legs} now follows by inserting the bounds \eqref{KM-LM} and one of the bounds in \eqref{KL-LL} into equations \eqref{pwthwbound2-for-Z}.

\end{proof}

\section{Proof of Proposition \ref{prop: Z=X in one leg}}
\label{s: Proof of Proposition Z=X in one leg}

Given a pack $\varpi=\left(\gamma; (\xi_j, u_j): 1\le j\le\gamma\right)$ \eqref{packdef}, and arbitrary $u_0, u_{\gamma+1}\in S^{2}$, let $\big((Y(t), \cX(t), Z(t)): 0\le t\le \theta\big)$ be the triplet of Markovian flight process, Lorentz exploration process and auxiliary $Z$-process jointly constructed with these data. We will prove the following bounds, stated in increasing order of difficulty/complexity. 
\begin{align}
\label{etamore}
&
\probab{\{\cX(t)\not\equiv Z(t): 0^-\le t\le \theta^+\} \cap \{\sum_{j=1}^{\gamma} \eta_j > 1\}}\le Cr^2 \abs{\log r}, 
\\
\label{etazero}
&
\probab{\{\cX(t)\not\equiv Z(t): 0^-\le t\le \theta^+\} \cap \{\sum_{j=1}^{\gamma} \eta_j = 0\}}\le Cr^2 \abs{\log r},
\\
\label{etaone}
&
\probab{\{\cX(t)\not\equiv Z(t): 0^-\le t\le \theta^+\} \cap \{\sum_{j=1}^{\gamma} \eta_j = 1\}}\le Cr^2 \abs{\log r}^2.
\end{align} 
Note that by construction $\eta_1=\eta_2=\eta_3=\eta_\gamma=0$, so the sums on the left hand side go actually from $4$ to $\gamma-1$ .
We stated and prove these bounds in their increasing order of complexity: 
\eqref{etamore} (proved in section \ref{ss: Proof of etamore}) and \eqref{etazero} (proved in section \ref{ss: Proof of etazero}) are of purely probabilistic nature while   \eqref{etaone} (proved in sections \ref{ss: Proof of etaone -- prep}-\ref{ss: Proof of etaone -- concl}) also relies on  the finer geometric understanding of the mismatch events $\wh\eta_j=1$ and $\wt\eta_j=1$. 

\subsection{Proof of \eqref{etamore}}
\label{ss: Proof of etamore}

This follows directly from Lemma \ref{lem:probab-corel-eta}. Indeed, given $\gamma$ and $\ueps=(\epsilon_j)_{1\le j\le \gamma}$, due to \eqref{corel-eta}, 
\begin{align*}
\condprobab{\sum_{j=1}^{\gamma} \eta_j > 1}{\ueps}
&
\le 
\gamma \max_{j}\condprobab{\eta_j=\eta_{j+1}=1}{\ueps}
+
\frac{\gamma^2}{2} \max_{j,k:\abs{j-k}>1} \condprobab{\eta_j=\eta_{k}=1}{\ueps}
\\
&
\le 
C\gamma r^2\abs{\log r} + C \gamma^2 r^2,  
\end{align*}
and hence, due to the exponential tail bound \eqref{gamma-theta-expo-tight} we get
\begin{align*}
\probab{\sum_{j=4}^{\gamma-1} \eta_j > 1}
= 
\expect{\condprobab{\sum_{j=4}^{\gamma-1} \eta_j > 1}{\ueps}}
\le 
Cr^2\abs{\log r},
\end{align*}
which concludes the proof of \eqref{etamore}. 
\qed

\subsection{Proof of \eqref{etazero}}
\label{ss: Proof of etazero}

First note that by construction of the processes $\big((\cX(t), Z(t)): 0^-< t < \theta^+\big)$ the following identities hold: 
\begin{align*}
&
\{\cX(t)\not\equiv Z(t): 0^-\le t\le \theta^+\} \cap \{\sum_{j=1}^{\gamma} \eta_j = 0\}
=
\{\cX(t)\not\equiv Y(t): 0^-\le t\le \theta^+\} \cap \{\sum_{j=1}^{\gamma} \eta_j = 0\},
\\
&
\{\cX(t)\not\equiv Y(t): 0^-\le t\le \theta^+\} 
=
\bigcup_{0<j<\gamma}
\left\{ \min_{\tau_j\le t\le \theta} \abs{Y^\prime_{j-1}-Y(t)} <r \right\}
\cup
\left\{ \min_{0\le t\le \tau_{j}} \abs{Y^\prime_{j+1}-Y(t)} <r \right\},
\end{align*}
and, hence
\begin{align}
\label{first}
&
\{\cX(t)\not\equiv Z(t): 0^-\le t\le \theta^+\} \cap \{\sum_{j=1}^{\gamma} \eta_j = 0\}
\\
\notag
&
\phantom{MM}
=
\bigcup_{0<j<\gamma}
\left(
\left\{ \min_{\tau_j\le t\le \tau_{j+1}} \abs{Y^\prime_{j-1}-Y(t)} <r \right\} 
\cup
\left\{ \min_{\tau_{j-1}\le t\le \tau_{j}} \abs{Y^\prime_{j+1}-Y(t)} <r \right\}
\right)
\cap\{\xi_j>1\}
\\
\notag
&
\phantom{MMMM}
\cup
\bigcup_{0<j<\gamma}
\left(
\left\{ \min_{\tau_{j+1}\le t\le \theta} \abs{Y^\prime_{j-1}-Y(t)} <r \right\}
\cup
\left\{ \min_{0\le t\le \tau_{j-1}} \abs{Y^\prime_{j+1}-Y(t)} <r \right\}
\right),
\\
\notag
&
\phantom{MM}
\subset
\bigcup_{0<j<\gamma}
\left(
\left\{ \min_{\tau_j\le t\le \tau_{j+1}} \abs{Y_{j-1}-Y(t)} <2r \right\} 
\cup
\left\{ \min_{\tau_{j-1}\le t\le \tau_{j}} \abs{Y_{j+1}-Y(t)} <2r \right\}
\right)
\cap\{\xi_j>1\}
\\
\notag
&
\phantom{MMMM}
\cup
\bigcup_{0<j<\gamma}
\left(
\left\{ \min_{\tau_{j+1}\le t\le \theta} \abs{Y_{j-1}-Y(t)} <2r \right\}
\cup
\left\{ \min_{0\le t\le \tau_{j-1}} \abs{Y_{j+1}-Y(t)} <2r \right\}
\right).
\end{align}
By simple geometric inspection we see
\begin{align*}
&
\left\{ \min_{\tau_j\le t\le \tau_{j+1}} \abs{Y_{j-1}-Y(t)} <2r \right\} 
\cap\{\xi_j>1\}
\subset
\left\{\angle(-u_{j-1},u_j)<4r \right\},
\\
&
\left\{ \min_{\tau_{j-1}\le t\le \tau_{j}} \abs{Y_{j+1}-Y(t)} <2r \right\}
\cap\{\xi_j>1\}
\subset
\left\{\angle(-u_{j+1},u_j)<4r \right\}.
\end{align*} 
And therefore, 
\begin{align}
\label{second}
\begin{aligned}
&
\max_{\ueps}
\condprobab{
\left\{ \min_{\tau_j\le t\le \tau_{j+1}} \abs{Y_{j-1}-Y(t)} <2r \right\} 
\cap
\{\xi_j>1\}}
{\ueps}
\le Cr^2
\\
&
\max_{\ueps}
\condprobab{
\left\{ \min_{\tau_{j-1}\le t\le \tau_{j}} \abs{Y_{j+1}-Y(t)} <2r \right\}
\cap
\{\xi_j>1\}}
{\ueps}
\le Cr^2.
\end{aligned}
\end{align}
On the other hand, from the conditional Green's function computations of section \ref{s: Naive}, in particular from Lemma \ref{lem:greenbounds-for-Y-conditional},  we get
\begin{align}
\label{third}
\begin{aligned}
&
\max_{\ueps}
\condprobab{
\min_{\tau_{j+1}\le t\le \theta} \abs{Y_{j-1}-Y(t)} <2r
}{\ueps}
\le 
\sup_{\ueps}
\condprobab{
\min_{\tau_2\le t <\infty } \abs{Y(t)} <2r
}{\ueps}
\le
C r^2 \abs{\log r}, 
\\
&
\max_{\ueps}
\condprobab{
\min_{0\le t\le \tau_{j-1}} \abs{Y_{j+1}-Y(t)} <2r
}{\ueps}
\le 
\sup_{\ueps}
\condprobab{
\min_{\tau_2\le t<\infty} \abs{Y(t)} <2r
}{\ueps}
\le 
C r^2 \abs{\log r}. 
\end{aligned}
\end{align}
Putting \eqref{first}, \eqref{second} and \eqref{third} together yields
\begin{align*}
\condprobab{\{\cX(t)\not\equiv Z(t): 0^-\le t\le \theta^+\} \cap \{\sum_{j=4}^{\gamma-1} \eta_j = 0\}
}{\ueps}
\le 
C\gamma r^2\abs{\log r}, 
\end{align*}
and hence, taking expectation over $\ueps$, we get \eqref{etazero}. 

\subsection{Proof of \eqref{etaone} -- preparations}
\label{ss: Proof of etaone -- prep}

Let $\gamma\in\{2\}\cup\{5,6,\dots\}$, and $\ueps=(\epsilon_j)_{1\le j\le\gamma}\in\{0,1\}^{\gamma}$ compatible with the definition of a pack, and $3<k < \gamma$ be fixed. Given a pack $\varpi$ with signature $\ueps$ we define yet another auxiliary process $\big( {Z}^{(k)}(t): 0^-< t <\theta^+ \big)$ as follows: 

\begin{enumerate}[$\circ$]

\item 
On $0^-<t\le \tau_{k-1}$, ${Z}^{(k)}(t)=Y(t)$.

\item
On $\tau_{k-1}<t \le \tau_{k}$, ${Z}^{(k)}(t)$ is constructed according to the rules of the $Z$-process, given in section \ref{ss: An auxiliary process}. 

\item
On $\tau_{k}<t<\theta^+$,   ${Z}^{(k)}(t)={Z}^{(k)}(\tau_k)+ Y(t)-Y(\tau_k)$. 

\end{enumerate}
\noindent
Note that on the event $\{\eta_j  = \delta_{j,k}: 1\le j\le \gamma\}$ we have ${Z}^{(k)}(t)\equiv Z(t)$, $0^-< t<\theta^+$. 

We will show that
\begin{align}
\label{will}
\begin{aligned}
&
\max_{\ueps, k}
\condprobab
{\{\cX(t)\not\equiv {Z}^{(k)}(t): 0^-\le t\le \theta^+\}
\cap
\{\eta_j  = \delta_{j,k}: 1\le j\le \gamma\}}
{\ueps}
\\
&\hskip4cm
\le
\max_{\ueps, k}
\condprobab
{\{\cX(t)\not\equiv {Z}^{(k)}(t): 0^-\le t\le \theta^+\}
\cap
\{\eta_k=1\}}
{\ueps}
\\
&\hskip4cm
\le 
C\gamma^2 r^2\abs{\log r}^2, 
\end{aligned}
\end{align}
and hence
\begin{align*}
&
\max_{\ueps}
\condprobab
{\{\cX(t)\not\equiv Z(t): 0^-\le t\le \theta^+\}
\cap
\{\sum_{k=1}^{\gamma}\eta_k  = 1\}}
{\ueps}
\\
&
\phantom{MMMMMMM}
\le 
\gamma
\max_{\ueps, k}
\condprobab
{\{\cX(t)\not\equiv Z(t): 0^-\le t\le \theta^+\}
\cap
\{\eta_j  = \delta_{j,k}: 1\le j\le \gamma\}}
{\ueps}
\\
&
\phantom{MMMMMMM}
\le 
C\gamma^3 r^2\abs{\log r}^2. 
\end{align*}
Then, taking expectation over $\ueps$ we get \eqref{etaone}. 

In order to prove \eqref{will} first write
\begin{align*}
&
\condprobab
{\{\cX(t)\not\equiv {Z}^{(k)}(t): 0^-\le t\le \theta^+\}
\cap
\{\eta_j  = \delta_{j,k}: 1\le j\le \gamma\}}
{\ueps}
\\
&
\phantom{MMMMMMMM}
\le 
\condprobab
{\{\cX(t)\not\equiv {Z}^{(k)}(t): 0^-\le t\le \theta^+\}
\cap
\{\eta_k  = 1\}}
{\ueps}
\\
&
\phantom{MMMMMMMM}
=
\condprobab
{\{\cX(t)\not\equiv {Z}^{(k)}(t): 0^-\le t\le \theta^+\}
\cap
\{\wh\eta_k  = 1\}
}
{\ueps}
+
\\
&
\phantom{MMMMMMMM}
\phantom{IIi} 
\condprobab
{\{\cX(t)\not\equiv {Z}^{(k)}(t): 0^-\le t\le \theta^+\}
\cap
\{\wt\eta_k  = 1\}
\cap
\{\wh\eta_k  = 0\}
}
{\ueps}, 
\end{align*}
and note that the three parts 
\begin{align}
\label{decomposition}
\begin{aligned}
&
\big({Z}^{(k)}(t): 0^- < t < \tau_{k-3} \big)
=
\big(Y(t): 0^- < t < \tau_{k-3} \big),  
\\
&
\big({Z}^{(k)}(\tau_{k-3}+t)-{Z}^{(k)}(\tau_{k-3}): 0\le t \le \tau_k-\tau_{k-3}\big),
\\
&
\big({Z}^{(k)}(\tau_k+t)-{Z}^{(k)}(\tau_k): 0\le t < \theta^+-\tau_k \big)
=
\big(Y(\tau_k+t)-Y(\tau_k): 0\le t < \theta^+-\tau_k \big),
\end{aligned}
\end{align} 
are \emph{independent}  -- even if the events $\{\wh\eta_k  = 1\}$, respectively, $\{\wt\eta_k  = 1\}\cap \{\wh\eta_k  = 0\}$ are specified.

From the construction of the processes $\big((\cX(t), Z^{(k)}(t)): 0^{-}< t < \theta^{+} \big)$ it follows that if $\big(Z^{(k)}(t): 0^{-}< t < \theta^{+}\big)$ is mechanically $r$-consistent then $\big(\cX(t) \equiv Z^{(k)}(t): 0^{-}< t < \theta^{+} \big)$.

Denote by $A^{(k)}_{a,a}$, $1\le a\le 3$, the event that the $a$-th part of the decomposition \eqref{decomposition} is mechanically \emph{$r$-inconsistent}, and by $A_{a,b}=A_{b,a}$, $1\le a,b \le 3$, $a\not=b$, the event that the $a$-th and $b$-th parts of the decomposition \eqref{decomposition} are mechanically \emph{$r$-incompatible}  -- in the sense of the definitions \eqref{consist} and \eqref{compat} in section \ref{ss: Mechanical consistency and compatibility}. In order to prove \eqref{will} we will have to prove appropriate upper bounds on the conditional probabilities
\begin{align}
\label{twelve bounds}
\begin{aligned}
&
\condprobab{\{\wh\eta_k  = 1\} \cap A^{(k)}_{a,b}}{\ueps}, 
\\
&
\condprobab{\{\wt\eta_k  = 1\} \cap \{\wh\eta_k  = 0\} \cap A^{(k)}_{a,b}}{\ueps}, 
\end{aligned}
\qquad\qquad\qquad\qquad
a,b=1,2,3.
\end{align}
These are altogether 12 bounds. However, some of them are formally very similar.

$A^{(k)}_{1,1}$,  $A^{(k)}_{3,3}$ and $A^{(k)}_{1,3}$ do not involve the middle part and therefore do not rely on the geometric arguments of the forthcoming sections \ref{ss: Geometric estimates}-\ref{ss: Proof of Corollary of main-geom}. Applying directly \eqref{probab-eta}, \eqref{ghbound-for-Y-conditional}, \eqref{KL-LL} and similar procedures as in section \ref{ss: Computation -- Naive}, without any new effort we get 
\begin{align}
\label{A11-A33-A13}
\begin{aligned}
&
\condprobab{\{\wh\eta_k  = 1\} \cap A^{(k)}_{a,b}}{\ueps }\le C\gamma^2 r^2, 
\\
&
\condprobab{\{\wt\eta_k  = 1\} \cap \{\wh\eta_k  = 0\} \cap A^{(k)}_{a,b}}{\ueps }\le C\gamma^2 r^2, 
\end{aligned}
\qquad\qquad\qquad
a,b=1,3.
\end{align}
We omit the repetition of these details. 

The remaining six bounds rely on the geometric arguments of sections \ref{ss: Geometric estimates}-\ref{ss: Proof of Corollary of main-geom} and, therefore, are postponed to section \ref{ss: Proof of etaone -- concl}.

\subsection{Geometric estimates}
\label{ss: Geometric estimates}

We analyse the middle segment of the process ${Z}^{(k)}$, presented in \eqref{decomposition}, restricted to the events $\{\wh\eta_k=1\}$, respectively, $\{\wt\eta_k=1\}\cap\{\wh\eta_k=0\}$. Since everything done in this analysis is invariant under time and space translations and also under rigid rotations of $\R^{3}$ it will be notationally convenient to place the origin of space-time at $(\tau_{k-2}, {Z}(\tau_{k-2}))$ and choose $u_{k-2}=e=(1,0,0)$, a fixed element of $S^{2}$. So, the ingredient random variables are $(\xi_-, u,\xi,v, \xi_+)$, fully independent and distributed as $\xi_-\sim EXP(1|\epsilon_{k-2})$, $\xi\sim EXP(1|\epsilon_{k-1})=EXP(1|1)$, $\xi_+ \sim EXP(1|\epsilon_{k})$, $u,v\sim UNI(S^{2})$. 

It will be enlightening to group the ingredient variables as $(\xi_-, (u,\xi,v), \xi_+)$, and accordingly write the sample space of this reduced context as $\R_+\times \D \times \R_+$, where $\D:=S^{2}\times \R_+ \times S^{2}$, with the probability measure $EXP(1|\epsilon_{k-2})\times \mu \times EXP(1|\epsilon_k)$ where, on $\D$, 
\begin{align}
\label{mu-meas}
\mu=UNI(S^{2})\times EXP(1|1) \times UNI(S^{2}). 
\end{align}
For $r<1$, let $\wh\sigma_r, \wt\sigma_r: \D\to \R_+\cup\{\infty\}$ be   
\begin{align*}
&
\wh\sigma_r(u,\xi,v)
:=
\inf\{t: \abs{\xi u + r\frac{u-v}{\abs{u-v}} +t e} < r\},
\\
&
\wt\sigma_r(u,\xi,v)
:=
\inf\{t: \abs{\xi u + r\frac{u-e}{\abs{u-e}} +t v} < r\}, 
\end{align*}
(with the usual convention $\inf\emptyset = \infty$), and  
\begin{align*}
\wh\A_r
:=
\{
(u,\xi,v)\in \D: 
\wh\sigma_r
<\infty
\}, 
\qquad\qquad
&
\wt\A_r
:=
\{
(u,\xi ,v)\in \D: 
\wt\sigma_r
<\infty\}.
\end{align*}
We define the process $\big(\wh{Z}_{r}(t): -\infty <t  <\infty\big)$ and  $\big(\wt{Z}_{r}(t): -\infty <t <\infty\big)$ in terms of $(u,\xi,v)\in\wh\A_r$, respectively, $(u,\xi,v)\in\wt\A_r$ as follows. Strictly speaking, these are \emph{deficient} processes, since $\mu(\wh\A_r)<1$, and $\mu(\wt\A_r)<1$. 

\begin{enumerate}[$\circ$]

\item
On $-\infty < t \le 0$, $\wh{Z}_{r}(t)= \wt{Z}_{r}(t)= te$.

\item
On $0\le t\le \xi$, $\wh{Z}_{r}(t)= \wt{Z}_{r}(t)= t u$, 

\item 
On $\xi \le t <\infty $, 
\begin{enumerate}[$\circ\circ$]

\item 
$\wh{Z}_{r}(t)=\wh{Z}_{r} (\xi) + (t-\xi) u$, 

\item 
$\wt{Z}_{r}(t)$ 
is the trajectory of a mechanical particle, with initial position 
$\wt{Z}_{r}(\xi)$ 
and initial velocity 
$\dot {\wt {Z}_{r}} (\xi^+)  = v$, 
bouncing elastically between two infinite-mass spherical scatterers centred at 
$r\frac{e-u}{\abs{e-u}}$, 
respectively, 
$\xi u + r\frac{u-v}{\abs{u-v}}$, 
and, eventually, flying indefinitely with constant terminal velocity. 

\end{enumerate} 
\end{enumerate}

The \emph{trapping time} $\wh\beta_r, \wt\beta_r \in\R_+$ and \emph{escape (terminal) velocity} $\wh w_r, \wt w_r\in S^{2}$ of the process $\wh{Z}_{r}(t)$, respectively, $\wt{Z}_{r}(t)$, are 
\begin{align}
\label{trapping-and-escape}
\begin{aligned}
&
\wh\beta_r:=
0, 
&&
\wh w_r := u, 
\\
&
\wt\beta_r:=
\sup\{ s<\infty : \dot {\wt{Z}_{r}}(\xi+s^+) \not = \dot {\wt{Z}_{r}}(\xi+s^-)\},
\qquad\qquad
\qquad\qquad
&&
\wt w_r := \dot {\wt{Z}_{r}}(\xi+\wt\beta_r^+).   
\end{aligned}
\end{align}
Note that $\wt\beta_r\ge \wt\sigma_r$. 

The relation of the middle segment of \eqref{decomposition} to $\wh{Z}_{r}$ and $\wt{Z}_{r}$ is the following: 
\begin{align}
\label{the relation of the middle}
\begin{aligned}
\left(
\{\wh\eta_k=1\}, 
\big({Z}^{(k)}(\tau_{k-2}+t)-{Z}^{(k)}(\tau_{k-2}): -\xi_{k-2}\le t\le \xi_{k-1}+\xi_{k}\big)
\right)
\sim
\hskip10mm
\\
\left(
\{\xi_->\wh\sigma_r \}, 
\big(\wh{Z}_{r}(t): -\xi_-\le t\le \xi+\xi_+\big)
\right),
\\[10pt]
\left(
\{\wh\eta_k=0\}\cap \{\wt\eta_k=1\}, 
\big({Z}^{(k)}(\tau_{k-2}+t)-{Z}^{(k)}(\tau_{k-2}): -\xi_{k-2}\le t\le \xi_{k-1}+\xi_{k}\big)
\right)
\sim
\hskip10mm
\\
\left(
\{\xi_-\le \wh\sigma_r \}\cap\{\xi_+> \wt\sigma_r \}, 
\big(\wt{Z}_{r}(t): -\xi_-\le t\le \xi+\xi_+\big)
\right),
\end{aligned}
\end{align}
where $\sim$ stands for equality in distribution. (Note, that the sequence signatures $(\epsilon_n)_{n\geq 1}$ is determined by the sequence of flight times $(\xi_n)_{n\geq 1}$.) So, in order to prove \eqref{will} we have to prove some subtle estimates for the processes $\wt{Z}_{r}$ amd $\wt{Z}_{r}$. The main estimates are collected in Proposition \ref{prop:main-geom} below

\begin{proposition}
\label{prop:main-geom}
There exists a constant $C<\infty$, such that for all $r<1$ and $s\in(0,\infty)$, the following bounds hold:
\begin{align}
\label{escape-angle-bound-shad}
\mu\left((u,h,v)\in \wh\A_r:  \angle(-e, \wh w_r)<s\right)
&
\le 
C r \min\{s, 1\},
\\
\label{escape-angle-bound-coll}
\mu\left((u,h,v)\in \wt\A_r: \angle(-e, \wt w_r)<s\right)
& 
\le 
C r \min\{s(\abs{\log s}\vee1), 1\},
\\
\label{trap-time-bound-coll}
\mu\left((u,h,v)\in \wt\A_r: 
r^{-1}\wt\beta_r>s\right)
&
\le 
C r \min\{s^{-1}(\abs{\log s}\lor1), 1\}.
\end{align}
\end{proposition}

\noindent
{\bf Remarks:}
The bound \eqref{escape-angle-bound-shad} is sharp in the sense that a lower bound of the same order can be proved. 
In contrast, we think that the upper bound in \eqref{escape-angle-bound-coll} is not quite sharp. However, it is sufficient for our purposes so we don't strive for a better estimate. 

The following consequence of Proposition \ref{prop:main-geom} will be used to prove \eqref{etaone}. 

\begin{corollary}
\label{cor: of main-geom}
There exists a constant $C<\infty$ such that the following bounds hold: 
\begin{align}
\label{85}
&
\condprobab{
\{\wh\eta_k=1\} 
\cap
\{\min_{\tau_{k-2}\le t \le \tau_{k}} 
\abs{ Z^{(k)}(t) - Z^{(k)}(\tau_{k-3})}< s\}}{\ueps}
\le 
C r s (\abs{\log s}\lor 1), 
\\
\label{85bis}
&
\condprobab{
\{\wh\eta_k=1\} 
\cap
\{\min_{\tau_{k-3}\le t \le \tau_{k-1}} 
\abs{Z^{(k)}(t)-Z^{(k)}(\tau_{k})}< s\}}{\ueps}
\le 
C r s (\abs{\log s}\lor 1), 
\\
\label{86}
&
\condprobab{
\{\wh\eta_k=0\}
\cap
\{\wt\eta_k=1\} 
\cap
\{\min_{\tau_{k-2}\le t \le \tau_{k}} 
\abs{ Z^{(k)}(t) - Z^{(k)}(\tau_{k-3})}< s\}}{\ueps}
\\
\notag
& 
\hskip93mm
\le 
C r \max\{s \abs{\log s}^2, r\abs{\log r}^2 \},
\\
\label{87}
& 
\condprobab{
\{\wh\eta_k=0\}\cap\{\wt\eta_k=1\} 
\cap
\{\min_{\tau_{k-3}\le t \le \tau_{k-1}+\wt\beta} 
\abs{Z^{(k)}(t)-Z^{(k)}(\tau_{k})}< s\}}{\ueps}
\\
\notag
& 
\hskip93mm
\le 
C r \max\{s \abs{\log s}^2, r\abs{\log r}^2 \}.
\end{align}
\end{corollary}

Proposition \ref{prop:main-geom} and its Corollary \ref{cor: of main-geom} are proved in sections \ref{ss: Proof of Proposition main-geom}, respectively, \ref{ss: Proof of Corollary of main-geom}. 

\subsection{Geometric estimates ctd: Proof of Proposition \ref{prop:main-geom}}
\label{ss: Proof of Proposition main-geom}

\subsubsection{Preparations}
\label{sss: Preparations}

Beside the probability measure $\mu$ (see \eqref{mu-meas}) we will also need the flat Lebesgue measure on $\D$, 
\begin{align*}
\lambda = UNI(S^{2})\times LEB(\R_+)\times UNI(S^{2}), 
\end{align*} 
so that 
\begin{align*}
d\mu(u,h,v) = \frac{e^{1-h}}{e-1}\one\{0\le h<1\} d\lambda(u,h,v). 
\end{align*}

For $r>0$ we define the \emph{dilation map} $D_r:\D\to\D$ as 
\begin{align*}
D_r(u,h,v)=(u, r h, v),
\end{align*}
and note that 
\begin{align*}
\wh\A_r
=
D_r \wh\A_1
\qquad\qquad
\wt\A_r 
= 
D_r\wt\A_1. 
\end{align*}
In the forthcoming steps all events in $\wh\A_r$ and  $\wt\A_r$ will be mapped by the inverse dilation  $D_r^{-1}=D_{r^{-1}}$ into  $\wh\A_1$, respectively, $\wt\A_1$. Therefore, in order to simplify notation we will use $\wh\A:=\wh\A_1$ and $\wt\A:=\wt\A_1$. 
 
The dilation $D_r$ transforms the measures $\mu$ as follows. Given an event $E\subset\D$, 
\begin{align}
\label{meas-dil}
\mu  (D_r E) 
=
\int_{D_r E} \frac{e^{1-h}}{e-1}\one\{0\le h\le 1\}d\lambda(u,h,v)
=
r
\int_{E} \frac{e^{1-rh}}{e-1}\one\{0\le h\le r^{-1}\}d\lambda(u,h,v), 
\end{align}
and hence, for any event $E\subset\D$ and any $\bar h <\infty$
\begin{align}
\label{bounds-on-mu-meas}
\frac{e^{1-r\bar h}}{e-1} r 
\lambda(E\cap\{h\le \bar h\})
\le
\mu(D_r E)
\le 
\frac{e}{e-1}r \lambda(E).
\end{align}

The following simple observation is of paramount importance in the forthcoming arguments:

\begin{proposition}
\label{prop: coll-and-shad-meas-is-finite}
In dimension $3$ (and more)
\begin{align}
\label{coll-and-shad-meas-is-finite}
\lambda(\wh\A)=\lambda(\wt\A)<\infty.
\end{align}
\end{proposition}

\begin{proof}
[Proof of Proposition \ref{prop: coll-and-shad-meas-is-finite}.]
Obviously, 
\begin{align*}
&
\wh\A
\subset
\wh\A^{\prime}
:=
\{
(u,h,v)\in \D: \angle(-e,u) \le  2 h^{-1}
\}, 
\\
&
\wt\A
\subset
\wt\A^{\prime}
:=
\{
(u,h,v)\in \D: \angle(-u,v) \le 2 h^{-1}
\}.
\end{align*}
Since, in dimension 3, 
\begin{align*}
\abs{\{(u,v)\in S^{2}\times S^{2}: \angle(-e,u)< 2 h^{-1}\}}
& =
\\
\abs{\{(u,v)\in S^{2}\times S^{2}: \angle(-u,v)< 2 h^{-1}\}}
& \le C \min\{h^{-2}, 1\},
\end{align*}
the claim follows by integrating over $h\in\R_+$. 
\end{proof}

\noindent
{\bf Remark:}
In $2$-dimension, the corresponding sets $\wh\A$, $\wt\A$ have infinite Lebesgue measure and, therefore, a similar proof would fail.

Due to \eqref{coll-and-shad-meas-is-finite} in 3-dimensions the following \emph{conditional probability measures} make sense
\begin{align*}
\lambda_{\wh\A}(\cdot)
=
\lambda(\cdot \,\big|\wh\A)
:=
\frac{\lambda(\cdot \cap \wh\A)}{\lambda(\wh\A)}, 
\qquad
\qquad
\lambda_{\wt\A}(\cdot)
=
\lambda(\cdot \,\big|\wt\A)
:=
\frac{\lambda(\cdot \cap \wt\A)}{\lambda(\wt\A)}, 
\end{align*}
and, moreover, due to \eqref{bounds-on-mu-meas} and \eqref{coll-and-shad-meas-is-finite}, for any event $E\in\D$
\begin{align*}
\lim_{r\to0} 
\mu(D_r E \,|\, \wh\A_r) 
=
\lambda_{\wh\A}(E), 
\qquad\qquad
\lim_{r\to0} 
\mu(D_r E \,|\, \wt\A_r) 
=
\lambda_{\wt\A}(E).
\end{align*}
In a technical sense, we will only use the upper bound in \eqref{bounds-on-mu-meas}, and \eqref{coll-and-shad-meas-is-finite}. 

In view of the upper bound in \eqref{bounds-on-mu-meas}, in order to prove \eqref{escape-angle-bound-shad}, \eqref{escape-angle-bound-coll} and  \eqref{trap-time-bound-coll} we need, in turn, 
\begin{align}
\label{escape-angle-bound-shad-transformed}
\lambda\left((u,h,v)\in \wh\A: \angle(-e, \wh w) \le s\right)
&
\le 
C \min\{s, 1\},
\\
\label{escape-angle-bound-coll-transformed}
\lambda\left((u,h,v)\in \wt\A: \angle(-e, \wt w) \le s\right)
&
\le 
C \min\{s(\abs{\log s} \lor 1), 1\}, 
\\
\label{trap-time-bound-coll-transformed}
\lambda\left((u,h,v)\in \wt\A: 
\wt\beta>s\right)
&
\le 
C \min\{s^{-1}(\abs{\log s} \lor 1),1\}. 
\end{align}
Here, and in the rest of this section, we use the simplified notation $\wh w:=\wh w_1$, $\wt w:=\wt w_1$, $\wt \beta:=\wt\beta_1$.  

\subsubsection{Proof of \eqref{escape-angle-bound-shad-transformed}}
\label{sss: Proof of escape-angle-bound-shad-transformed}

\begin{proof}
This is straightforward. Recall \eqref{trapping-and-escape}: $\wh w(u,h,v)=u$. For easing notation let 
\begin{align*}
\vartheta:= \angle(-e, u)
\end{align*} 
and note that for any $t\in\R_+$
\begin{align*}
\abs{\{u\in S^{2} : 0\le \vartheta\le t\}}
\le 
C\min\{t^2,1\},
\end{align*}
with some explicit $C<\infty$. 

Then, 
\begin{align*}
&
\lambda\left((u,h,v)\in \wh\A: \angle(-e, \wh w) \le s\right)
\le 
\lambda\left((u,h,v)\in \wh\A^\prime: \vartheta \le s\right)
\\
&
\hskip10mm
\le
\lambda\left((u,h,v)\in \D: \vartheta \le \min\{s, 2h^{-1}\}\right)
\\
&
\hskip10mm
=
\lambda\left((u,h,v)\in \D: 
\{h\le 2 s^{-1}\} \cap \{\vartheta \le s\}
\right)
+
\lambda\left((u,h,v)\in \D: 
\{h\ge 2 s^{-1}\}
\cap
\{\vartheta \le 2 h^{-1}
\}
\right)
\\
&
\hskip10mm
\le 
C s.
\end{align*}

\end{proof}

\subsubsection{Proof of \eqref{escape-angle-bound-coll-transformed} and \eqref{trap-time-bound-coll-transformed}}
\label{sss: Proof of trap-time-bound-coll-transformed and escape-angle-bound-coll-transformed}

Figure 3 aides understanding this subsection. 


\begin{figure}[ht!]
  \begin{center}    
    \includegraphics[width=0.9\textwidth]{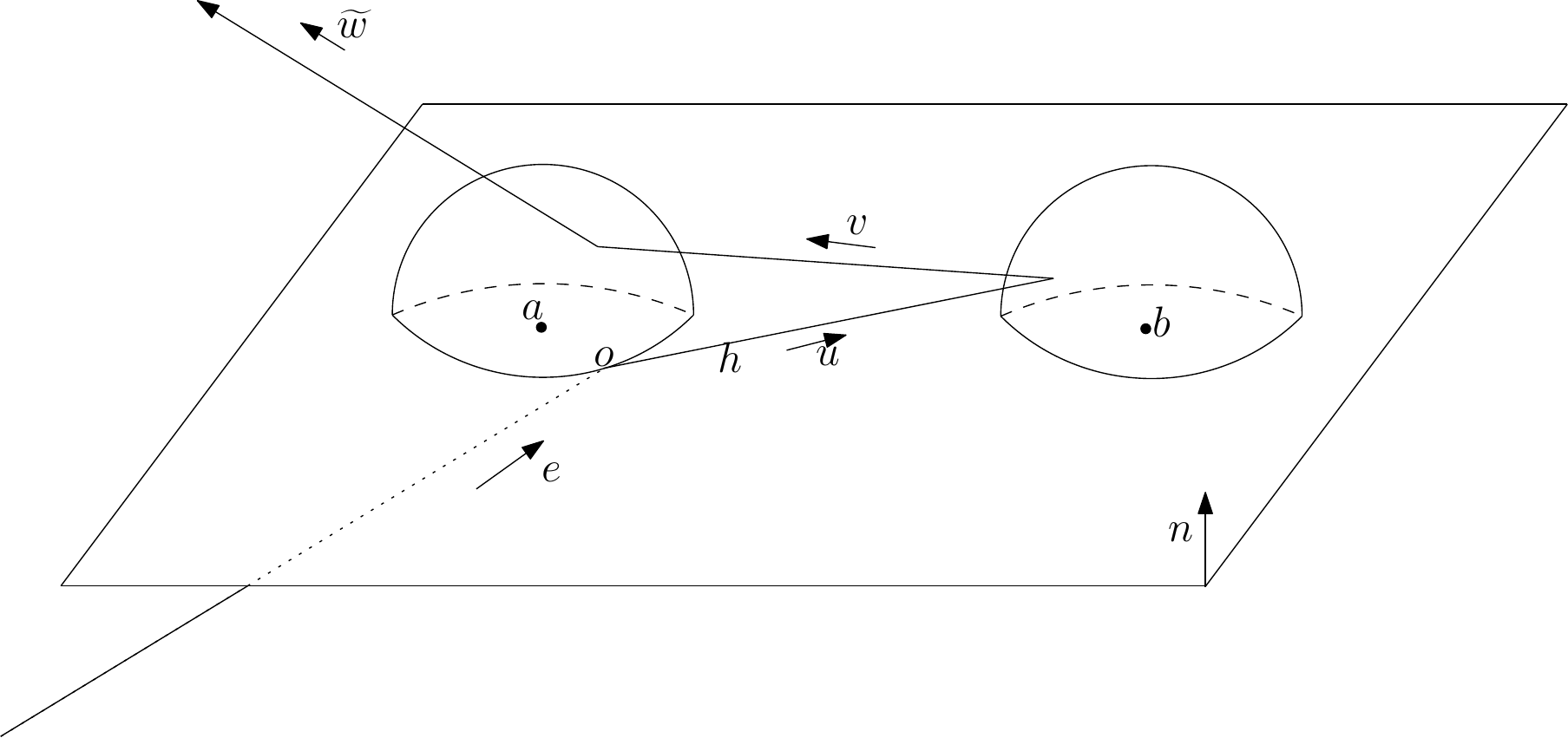}
  \end{center}
  \caption{%
    {\tt
    Above we show a 3 dimensional example of the geometric labelling used in this section. The $Z$ trajectory enters with velocity $e$ from beneath the relevant plane (the dotted line represents motion below the plane). After which the particle remains above the plane. }
  }%
  
  \label{fig:recollision_plane}

\end{figure}

Let $a$ and $b$ be the vectors in $\R^{3}$ pointing from the origin to the centre of the spherical  scatterers of radius 1, on which the first, respectively, the second collision occurs: 
\begin{align*}
a= \frac{e-u}{\abs{e-u}}, 
\qquad\qquad
b=hu + \frac{u-v}{\abs{u-v}},
\end{align*}  
and $n$ the unit vector orthogonal to the plane determined by $a$ and $b$, pointing so, that $e\cdot n>0$: 
\begin{align*}
n:= \frac{a\times b}{\abs{a}\abs{b} \sin(\angle(a,b))}, 
\end{align*}
with 
\begin{align}
\label{times}
&
a\times b
=
(h+\frac{1}{\abs{u-v}})\frac{1}{\abs{e-u}} \ e\times u
-
\frac{1}{\abs{e-u}\abs{u-v}} \ e\times v
+\frac{1}{\abs{e-u}\abs{u-v}} \ u\times v, 
\\[10pt]
\label{absinbound}
&
\abs{a}=1, 
\qquad\qquad\qquad
h-1 \le \abs{b} \le h+1,
\qquad\qquad\qquad
0\le
\sin(\angle(a,b))\le1.
\end{align}

Assume there are altogether $\nu\ge 3$ collisions (which occur alternatively, on the first and second scatterer) before escape. Let $w_0=e$ and $w_j$, $1\le j\le\nu$,  the outgoing velocity after the $j$-th scattering. So, $w_1=u, w_2=v, \dots, w_\nu=\wt w$. 

The proof of \eqref{escape-angle-bound-coll-transformed}  and \eqref{trap-time-bound-coll-transformed} relies on the following observations: 

\begin{enumerate}[(a)]

\item
\label{proj-incr}
The $n$-projection of the velocity of the moving particle does not decrease. More precisely, for $1\le j\le\nu$, $0\le w_{j-1} \cdot n \le w_j\cdot n$. This is due to the choice of the plane determined by the centres of the two scatterers and the first impact point.

\item
\label{angles-on-two-sides}
Since $e\cdot n>0$ and $w_j\cdot n>0$, for all $1\le j\le \nu$ we have
$\angle(-e, w_j)> \frac{\pi}{2}-\angle(n, w_j)$.

\item
\label{time-in-slab}
The trapping time $\wt \beta$ is certainly not longer than the time the moving particle spends in the slab $\{ x\in\R^{3}: 0\le x\cdot n \le 1\}$. Moreover, the scatterers are defocusing, that is, each time there is a collision the velocity component in the $n$ direction increases. Therefore, it follows that 
\begin{align}
\label{which-is-better}
\wt \beta \le h+ \abs{v\cdot n}^{-1}. 
\end{align}

\end{enumerate}

\begin{proof}
[Proof of \eqref{escape-angle-bound-coll-transformed}]
Without loss of generality we may assume $s\le \frac{\pi}{2}$. 

From the arguments \ref{proj-incr} and \ref{angles-on-two-sides} above it follows, in particular, that   
\[
\angle(-e, \wt w)
= 
\angle(-e,  w_\nu)
\ge 
\frac{\pi}{2}-\angle(n, w_\nu)
\ge 
\frac{\pi}{2}-\angle(n, w_2)
=
\frac{\pi}{2}-\angle(n, v),
\]
and hence
\begin{align}
\label{hence}
\lambda\left((u,h,v)\in \wt\A: 
\angle(-e, \wt w) < s\right)
\le 
\lambda\left((u,h,v)\in \wt\A^{\prime}: 
\abs{n\cdot v} < 2 s\right).
\end{align}
Note that due to \eqref{times} and \eqref{absinbound}
\begin{align*}
\abs{v\cdot n}
\ge 
\frac12 \abs{v\cdot(e\times u)},
\end{align*}
and thus 
\begin{align}
\label{thus}
\lambda\left((u,h,v)\in \wt\A^\prime: 
\abs{v\cdot n}<2s\right)
\le 
\lambda\left((u,h,v)\in \wt\A^\prime: 
\abs{e\cdot(u\times v)}<4s\right). 
\end{align}
Next, if $u$ and $v$ are i.i.d. $UNI(S^{2})$-distributed then 
\begin{align*}
w:=\frac{u\times v}{\abs{u\times v}}, 
\qquad
\text{ and }
\qquad
\vartheta:=\abs{u\times v}=\sin(\angle(u,v))
\end{align*}
are independent and distributed as
\begin{align*}
w\sim UNI(S^{2}), 
\qquad
\qquad
\vartheta \sim \ind{0\le t\le 1} (1-t^2)^{-1/2} t dt. 
\end{align*}
Therefore, 
\begin{align}
\notag
\lambda\left((u,h,v)\in \wt\A^\prime: 
\abs{e\cdot(u\times v)}<4s\right)
&
=
\int_0^\infty dh
\int_{S^{2}} dw
\int_0^{\min\{2/h, 1\}} (1-t^2)^{-1/2}t dt
\one\{\abs{e\cdot w}\le \frac{4s}{t}\}
\\
\notag
&
=
\int_0^\infty dh
\int_0^{\min\{2/h, 1\}} (1-t^2)^{-1/2}dt
\min\{4s, t\}
\\
\label{second-in-split}
&
\le 
C \min\{ s \abs{\log s}\lor 1), 1\}.
\end{align}
The last step follows from explicit computations which we omit. 

Finally, \eqref{hence}, \eqref{thus} and \eqref{second-in-split} yield \eqref{escape-angle-bound-coll-transformed}.
\end{proof}

\begin{proof}
[Proof of \eqref{trap-time-bound-coll-transformed}]

We proceed with the bound \eqref{which-is-better}: 
\begin{align}
\label{split}
\lambda\left((u,h,v)\in \wt\A: 
\wt\beta>s\right)
\le 
\lambda\left((u,h,v)\in \wt\A^\prime: 
h>\frac{s}{2}\right)
+
\lambda\left((u,h,v)\in \wt\A^\prime: 
\abs{v\cdot n}<\frac{2}{s}\right).
\end{align}
Bounding the first term on the right hand side of \eqref{split} is straightforward: 
\begin{align}
\notag
\lambda\left((u,h,v)\in \wt\A^\prime: 
h>\frac{s}{2}\right)
&
=
\int_{s/2}^\infty
\abs{\{(u,v)\in S^{2}\times S^{2}: \angle(-u,v)<2h^{-1}\}}dh
\\
\label{first-in-split}
&
\le 
C \int_{s/2}^\infty \min\{h^{-2}, 1\} dh 
\le 
C\min\{s^{-1}, 1\}.
\end{align}
Concerning the second term on the right hand side of \eqref{split}, this has exactly been done in the proof of \eqref{escape-angle-bound-coll-transformed} above, ending in \eqref{second-in-split} -- with the r\^ole of $s$ and $s^{-1}$ swapped.

\eqref{split}, \eqref{first-in-split} and \eqref{second-in-split} yield \eqref{trap-time-bound-coll}.

\end{proof}

\subsection{Geometric estimates ctd: Proof of Corollary \ref{cor: of main-geom}}
\label{ss: Proof of Corollary of main-geom}

We start with the following straightforward geometric fact. 

\begin{lemma}
\label{lem: time-spent-close}
Let $e, w \in S^{2}$ and $x\in \R^{3}$. 
Then 
\begin{align}
\label{time-spent-close}
\begin{aligned}
\abs{\{t^\prime>0: \min_{t\ge 0}\abs{x+t^{\prime} w + t e}< s\}}
=
\abs{\{t^\prime>0: \min_{t\ge 0}\abs{x+t w + t^{\prime} e}< s\}}
\le 
\frac{4s}{\angle(-e,w)}. 
\end{aligned}
\end{align} 
\end{lemma}

\begin{proof}
[Proof of Lemma \ref{lem: time-spent-close}]
This is elementary $3$-dimensional geometry. We omit the details. 
\end{proof}

\begin{proof}
[Proof of \eqref{85} and \eqref{85bis}] 
On $\{\wh\eta_k=1\}$ 
\begin{align}
\label{valami}
\begin{aligned}
&
\min_{\tau_{k-2}\le t \le \tau_{k}} 
\abs{ Z^{(k)}(t) - Z^{(k)}(\tau_{k-3})}
\ge 
\min_{0\le t}
\abs{t u_{k-1} +  \xi_{k-2}u_{k-2}} 
\\
&
\min_{\tau_{k-3}\le t \le \tau_{k-1}} 
\abs{Z^{(k)}(t)-Z^{(k)}(\tau_{k})}
\ge 
\min\{\min_{0\le t}
\abs{\xi_{k-1} u_{k-1}   + t u_{k-2} + \xi_k u_{k-1}},
\xi_{k} 
\}.
\end{aligned}
\end{align}
The bounds in \eqref{85} and \eqref{85bis} follow from applying \eqref{time-spent-close} and \eqref{escape-angle-bound-shad}, bearing in mind that the distribution density of $\xi_{k-2}$ and $\xi_{k}$ is bounded.  Since these are very similar we will only prove \eqref{85} here.
\begin{align*}
&
\probab{
\{\wh\eta_k=1\} 
\cap
\{\min_{\tau_{k-2}\le t \le \tau_{k}} 
\abs{ Z^{(k)}(t) - Z^{(k)}(\tau_{k-3})}< s\}}
\\
&\hskip4cm
\le 
\probab{
\{\wh\eta_k=1\} 
\cap
\{
\min_{t \ge 0}
\abs{t u_{k-1} +  \xi_{k-2}u_{k-2}} <s
\}
}
\\
&\hskip4cm
=
\int_{\wh\A_r}
\probab{\xi_-\in \{t^{\prime}: \min_{t\ge0}\abs{t u + t^{\prime} e}<s\}} d\mu(u,h,v)
\\
&\hskip4cm
\le 
C
\int_{\wh\A_r}
\min\{\frac{s}{\angle(-e,u)}, 1\}
d\mu(u,h,v)
\\
&\hskip4cm
\le 
C rs (\abs{\log s}\lor 1).
\end{align*}
In the first step we used \eqref{valami}. 
The second step follows from the representation \eqref{the relation of the middle}. 
The third step relies on \eqref{time-spent-close} and on uniform boundedness of the distribution density of $\xi_-$ (which is either $EXP(1|1)$ or $EXP(1|0)$, depending on the value of $\epsilon_{k-2}$). 
Finally, the last calculation is based on \eqref{escape-angle-bound-shad}. 
  
\end{proof}

\begin{proof}
[Proof of \eqref{86}]
 
\begin{align}
\label{infmininf in 86}
&
\min_{\tau_{k-2}\le t \le \tau_{k}} 
\abs{ Z^{(k)}(t) - Z^{(k)}(\tau_{k-3})}
\\
\notag
&\hskip2cm
=
\min\left\{
\min_{\tau_{k-2}\le t \le \tau_{k-1}+\wt\beta} 
\abs{ Z^{(k)}(t) - Z^{(k)}(\tau_{k-3})}, 
\min_{\tau_{k-1}+\wt\beta \le t \le \tau_{k}} 
\abs{ Z^{(k)}(t) - Z^{(k)}(\tau_{k-3})}
\right\}.
\end{align}
Here, and in the rest of this proof, $\wt\beta$ and $\wt w$ denote the trapping time and escape direction of the recollision sequence: 
\begin{align*}
\wt\beta:=
\max\{s\le\xi_k: \dot{Z}^{(k)}(\tau_{k-1}+s^-)\not=\dot{Z}^{(k)}(\tau_{k-1}+s^+)\}
\qquad\qquad
\wt w:= \dot{Z}^{(k)}(\tau_{k-1}+\wt\beta^+).
\end{align*}

To bound the first expression on the right hand side of \eqref{infmininf in 86} we first observe that by the triangle inequality
\begin{align}
\label{triangle-1}
\min_{\tau_{k-2}\le t \le \tau_{k-1}+\wt\beta} 
\abs{ Z^{(k)}(t) - Z^{(k)}(\tau_{k-3})}
\ge 
\xi_{k-2} - \xi_{k-1} - 4r.
\end{align}
Applying the representation and bounds developed in sections \ref{ss: Geometric estimates}, \ref{ss: Proof of Proposition main-geom},
\begin{align}
\notag
&
\probab{
\{\wh\eta_k=0\}\cap\{\wt\eta_k=1\} 
\cap
\{\min_{\tau_{k-2}\le t \le \tau_{k-1}+\wt\beta} 
\abs{ Z^{(k)}(t) - Z^{(k)}(\tau_{k-3})}<s\}}
\\
\notag
&
\hskip5cm
\le
\probab{
\{\wh\eta_k=0\}\cap\{\wt\eta_k=1\} 
\cap
\{\xi_{k-2}\le \xi_{k-1}+4r+s\}}
\\
\notag
&\hskip5cm
=
\int_{\wt\A_r}\probab{\xi_- < h+ 4r + s} d\mu(u,h,v)
\\
\notag
&\hskip5cm
\le 
C
\int_{\wt\A_r}(\min\{h, 1\}+ 4r + s) d\mu(u,h,v)
\\
\label{first in 86}
&\hskip5cm
\le 
Cr^2 +C rs + C r^2\abs{\log r}. 
\end{align}
In the first step we used \eqref{triangle-1}. 
The second step follows from the representation \eqref{the relation of the middle}. 
The third step relies on uniform boundedness of the distribution density of $\xi_-$ (which is either $EXP(1|1)$ or $EXP(1|0)$, depending on the value of $\epsilon_{k-2}$). 
Finally, the last step follows from explicit calculation, using \eqref{bounds-on-mu-meas}.

To bound the second term on the right hand side of \eqref{infmininf in 86} we proceed as in the proof of \eqref{85} above. First note that 
\begin{align}
\label{valami mas}
\min_{\tau_{k-1}+\wt\beta \le t \le \tau_{k}} 
\abs{ Z^{(k)}(t) - Z^{(k)}(\tau_{k-3})}
\ge 
\min_{0\le t}
\abs{(Z^{(k)}(\tau_{k-2})-Z^{(k)}(\tau_{k-1}+\wt\beta) ) +t\wt w +\xi_{k-2}u_{k-2}}.
\end{align}
Using in turn \eqref{valami mas}, \eqref{the relation of the middle}, \eqref{time-spent-close} and uniform boundedness of the distribution density of $\xi_-$ (which is either $EXP(1|1)$ or $EXP(1|0)$, depending on the value of $\epsilon_{k-2}$), and finally  \eqref{escape-angle-bound-coll}, we obtain: 
\begin{align}
\notag
&
\probab{
\{\wh\eta_k=0\}\cap\{\wt\eta_k=1\} 
\cap
\min_{\tau_{k-1}+\wt\beta \le t \le \tau_{k}} 
\abs{ Z^{(k)}(t) - Z^{(k)}(\tau_{k-3})}
<s
}
\\
\notag
&\hskip1cm
\le 
\probab{
\{\wh\eta_k=0\}\cap\{\wt\eta_k=1\} 
\cap
\{
\min_{0\le t}
\abs{(Z^{(k)}(\tau_{k-2})-Z^{(k)}(\tau_{k-1}+\wt\beta) ) +t\wt w +\xi_{k-2}u_{k-2}}<s
\}
}
\\
\notag
&\hskip1cm
=
\int_{\wt\A_r}
\probab{\xi_-\in \{t^{\prime}: 
\min_{0\le t}\abs{\wt{Z}_{r}(\wt\beta_r)+ t \wt w_r + t^{\prime} e}<s\}} d\mu(u,h,v)
\\
\notag
&\hskip1cm
\le 
C
\int_{\wt\A_r}
\min\{\frac{s}{\angle(-e,\wt w_r)}, 1\}
d\mu(u,h,v)
\\
\notag
&\hskip1cm
\le 
Crs\left(\int_{s}^\infty \frac{\log(x)}{x} dx \lor 1 \right)
\\
\label{second in 86}
&\hskip1cm
\le 
C rs (\abs{\log s}^2\lor 1).
\end{align}
In the second last line we use \eqref{escape-angle-bound-coll} and integrate by parts. From \eqref{infmininf in 86}, \eqref{first in 86} and \eqref{second in 86} we obtain \eqref{86}. 
\end{proof}

\begin{proof}
[Proof of \eqref{87}] 
We proceed very similarly as in the proof of \eqref{86}. 
\begin{align}
\label{infmininf in 87}
&
\min_{\tau_{k-3}\le t \le \tau_{k-1}+\wt\beta} 
\abs{ Z^{(k)}(t) - Z^{(k)}(\tau_{k})}
\\
\notag
&\hskip2cm
\ge
\min\left\{
\min_{\tau_{k-2}\le t \le \tau_{k-1}+\wt\beta} 
\abs{ Z^{(k)}(t) - Z^{(k)}(\tau_{k})}, 
\min_{\tau_{k-3}\le t \le \tau_{k-2}} 
\abs{ Z^{(k)}(t) - Z^{(k)}(\tau_{k})}
\right\}.
\end{align}
To bound the first expression on the right hand side of \eqref{infmininf in 87} we first observe that by the triangle inequality
\begin{align}
\label{triangle-2}
\min_{\tau_{k-2}\le t \le \tau_{k-1}+\wt\beta} 
\abs{ Z^{(k)}(t) - Z^{(k)}(\tau_{k})}
\ge 
\xi_k- 2\wt\beta-4r.
\end{align}
Using in turn \eqref{triangle-2}, \eqref{the relation of the middle}, \eqref{trap-time-bound-coll} and explicit computation based on uniform boundedness of the distribution density of $\xi_+$ (which is either $EXP(1|1)$ or $EXP(1|0)$, depending on the value of $\epsilon_{k}$) we write
\begin{align}
\notag
&
\probab{
\{\wh\eta_k=0\}
\cap
\{\wt\eta_k=1\}
\cap
\{\min_{\tau_{k-2}\le t \le \tau_{k-1}+\wt\beta} 
\abs{ Z^{(k)}(t) - Z^{(k)}(\tau_{k})}<s\}}
\\
\notag
&
\hskip15mm
\le
\probab{
\{\wh\eta_k=0\}
\cap
\{\wt\eta_k=1\}
\cap
\{\xi_k<8r+2s\}}
+
\probab{
\{\wh\eta_k=0\}
\cap
\{\wt\eta_k=1\}
\cap
\{\xi_k<4\wt\beta\}}
\\
\notag
&
\hskip15mm
=
\probab{\xi_+<8r+2s}\mu(\wt\A_r)
+
\expect{\mu( (u,h,v) \in \wt\A_r : \xi_+ \le 4 \wt\beta_r )}
\\
\notag
&
\hskip15mm
\le 
Cr(r+s)
+
C r \expect{\min\{\left(\frac{\xi_+}{2r}\right)^{-1}\left(\abs{\log \frac{\xi_+}{2r}} \lor 1\right), 1\}} 
\\
\label{first in 87}
&
\hskip15mm
\le 
Cr^2 +C rs + C r^2 \abs{\log r}^2. 
\end{align}

The second term on the right hand side of \eqref{infmininf in 87} is bounded in a very similar way as the analogous second term on the right hand side of \eqref{infmininf in 86}, see \eqref{valami mas}-\eqref{second in 86}. Without repeating these details we state that 
\begin{align}
\label{second in 87}
\probab{
\{\wh\eta_k=0\}
\cap
\{\wt\eta_k=1\}
\cap
\{\min_{\tau_{k-3}\le t \le \tau_{k-2}} 
\abs{ Z^{(k)}(t) - Z^{(k)}(\tau_{k})}
<s\}}
\le
Crs\abs{\log s}^2.
\end{align}
Eventually, from \eqref{infmininf in 87}, \eqref{first in 87} and \eqref{second in 87} we obtain \eqref{87}. 

\end{proof}

\subsection{Proof of \eqref{etaone} -- concluded}
\label{ss: Proof of etaone -- concl}

Recall the events $A^{(k)}_{a,b}$, $a,b\in\{1,2,3\}$ from the end of section \ref{ss: Proof of etaone -- prep}. 

The bounds \eqref{85}, \eqref{85bis}, respectively, \eqref{86}, \eqref{87}, with $s=r$,  directly imply
\begin{align}
\label{A22}
\begin{aligned}
&
\condprobab{
\{\wh\eta_k  = 1\} \cap A^{(k)}_{2,2}}{\ueps }
\le 
C\gamma r^2\abs{\log r}, 
\\
&
\condprobab{\{\wt\eta_k  = 1\} \cap \{\wh\eta_k  = 0\} \cap A^{(k)}_{2,2}}{\ueps }
\le 
C\gamma r^2\abs{\log r}^2. 
\end{aligned}
\end{align}

It remains to prove 
\begin{align}
\label{A12-A32}
\begin{aligned}
&
\condprobab{
\{\wh\eta_k  = 1\} \cap A^{(k)}_{b,2}}{\ueps }
\le 
C\gamma r^2\abs{\log r}, 
\\
&
\condprobab{\{\wt\eta_k  = 1\} \cap \{\wh\eta_k  = 0\} \cap A^{(k)}_{b,2}}{\ueps }
\le 
C\gamma r^2\abs{\log r}^2,
\end{aligned}
\qquad\qquad\qquad\qquad
b=1,3.
\end{align}
Since the cases $b=1$ and $b=3$ are formally identical we will go through the steps of proof with $b=3$ only. In order to do this we first define the necessary occupation time measures (Green's functions). For $A\subset\R^3$, define the following occupation time measures for the last part of \eqref{decomposition}
\begin{align*}
G^{(k)}_{\ueps}(A)
:=
&
\condexpect
{\#\{1\le j\le\gamma-k: Y(\tau_{j})\in A\} }
{\eps_{k+j}: 1\le j\le\gamma-k}
\\
=
&
\condexpect
{\#\{k+1\le j\le\gamma : Z^{(k)}(\tau_{j})-Z^{(k)}(\tau_{k})\in A\} }
{\ueps \cap \{\wh\eta_k=1\}}
\\
=
&
\condexpect
{\#\{k+1\le j\le\gamma : Z^{(k)}(\tau_{j})-Z^{(k)}(\tau_{k})\in A\} }
{\ueps \cap \{\wt\eta_k=1\}\cap \{\wh\eta_k=0\}},
\\
H^{(k)}_{\ueps}(A)
:=
&
\condexpect
{\abs{\{0\le t\le \tau_{\gamma-k}: Y(t)\in A\}} }
{\eps_{k+j}: 1\le j\le\gamma-k}
\\
=
&\condexpect
{\abs{\{\tau_k\le t\le\theta : Z^{(k)}(t)-Z^{(k)}(\tau_{k})\in A\} }}
{\ueps \cap \{\wh\eta_k=1\}}
\\
=
&
\condexpect
{\abs{\{\tau_k\le t\le\theta : Z^{(k)}(t)-Z^{(k)}(\tau_{k})\in A\} }}
{\ueps \cap \{\wt\eta_k=1\}\cap \{\wh\eta_k=0\}}.
\end{align*}
Similarly, define the following occupation time measures for the middle part of \eqref{decomposition}
\begin{align*}
&
\wh G^{(k)}_{\ueps}(A)
:=
\condexpect
{\#\{1\le j\le3: Z^{(k)}(\tau_{k-j})-Z^{(k)}(\tau_{k})\in A\} \cdot \wh\eta_k}
{\ueps}
\\
&
\wh H^{(k)}_{\ueps}(A)
:=
\condexpect
{\abs{\{\tau_{k-3}\le t\le \tau_k: Z^{(k)}(t)-Z^{(k)}(\tau_{k})\in A\}} \cdot \wh\eta_k}
{\ueps}
\\
&
\wt G^{(k)}_{\ueps}(A)
:=
\condexpect
{\#\{1\le j\le3: Z^{(k)}(\tau_{k-j})-Z^{(k)}(\tau_{k})\in A\} \cdot \wt\eta_k \cdot (1-\wh\eta_k)}
{\ueps}
\\
&
\wt H^{(k)}_{\ueps}(A)
:=
\condexpect
{\abs{\{\tau_{k-3}\le t\le \tau_k: Z^{(k)}(t)-Z^{(k)}(\tau_{k})\in A\}} \cdot \wt\eta_k \cdot (1-\wh\eta_k)}
{\ueps}. 
\end{align*}
Using the independence of the middle and last parts in the decomposition \eqref{decomposition}, similarly as \eqref{pwhtbound-for-Y} or \eqref{pwhtbound-for-Z}, following bounds are obtained 
\begin{align}
\label{bound for A23}
\begin{aligned}
&
\condprobab{
\{\wh\eta_k  = 1\} \cap A^{(k)}_{3,2}}{\ueps }
\le 
Cr^{-1}\int_{\R^3}
G^{(k)}_{\ueps} (B_{x, 2r}) \wh H^{(k)}_{\ueps} (dx)
+
Cr^{-1}\int_{\R^3}
H^{(k)}_{\ueps} (B_{x, 3r}) \wh G^{(k)}_{\ueps} (dx),
\\
&
\condprobab{
\{\wt\eta_k  = 1\} \cap \{\wh\eta_k  = 0\} \cap A^{(k)}_{3,2}}{\ueps }
\le 
\\
&
\hskip38mm
\le 
Cr^{-1}\int_{\R^3}
G^{(k)}_{\ueps} (B_{x, 2r}) \wt H^{(k)}_{\ueps} (dx)
+
Cr^{-1}\int_{\R^3}
H^{(k)}_{\ueps} (B_{x, 3r}) \wt G^{(k)}_{\ueps} (dx).
\end{aligned}
\end{align}
Due to \eqref{GHbound-for-Y-conditional} of Lemma \ref{lem:greenbounds-for-Y-conditional} by direct computations the following upper bounds hold
\begin{align}
\label{grbo-loc}
&
G^{(k)}_{\ueps} (B_{x,2r})
\le
C F(\abs{x}), 
&&
H^{(k)}_{\ueps} (B_{x,3r})
\le
C F(\abs{x}), 
\end{align}
where $C<\infty$ is an appropriately chosen constant and $F:\R_+\to\R$, 
\begin{align*}
F(u)
:=
r
\one\{0\le u< r\}
+
\frac{r^3}{u^2}
\one\{r\le u< 1\}
+
\frac{r^3}{u}
\one\{1\le u< \infty\}.
\end{align*}
On the other hand, from \eqref{85}, \eqref{85bis}, \eqref{86}, \eqref{87} of  Corollary \ref{cor: of main-geom} follows that 
\begin{align}
\label{grbooooooo-loc}
\begin{aligned}
&
\wh G^{(k)}_{\ueps} (B_{0,s})
\le 
Crs (\abs{\log s} \lor 1), 
&&
\wh H^{(k)}_{\ueps} (B_{0,s})
\le 
Crs (\abs{\log s} \lor 1), 
\\
&
\wt G^{(k)}_{\ueps} (B_{0,s})
\le 
Cr \max\{s\abs{\log s}^2, r\abs{\log r}^2\}
&&
\wt H^{(k)}_{\ueps} (B_{0,s})
\le 
Cr \max\{s\abs{\log s}^2, r\abs{\log r}^2\}.
\end{aligned}
\end{align}
Finally, we also have the global bounds
\begin{align}
\label{grboo-glob}
\begin{aligned}
&
\wh G^{(k)}_{\ueps} (\R^3)=
3\condexpect{\wh\eta_k}{\ueps}
\le Cr, 
&&
\wh H^{(k)}_{\ueps} (\R^3)
=
\condexpect{\wh\eta_k\cdot \sum_{j=k-2}^{k} \xi_j}{\ueps}
\le Cr, 
\\
&
\wt G^{(k)}_{\ueps} (\R^3)=
3\condexpect{\wt\eta_k \cdot(1-\wh\eta_k)}{\ueps}
\le Cr, 
&&
\wt H^{(k)}_{\ueps} (\R^3)
=
\condexpect{\wt\eta_k \cdot(1-\wh\eta_k)\cdot \sum_{j=k-2}^{k} \xi_j}{\ueps}
\le Cr.
\end{aligned}
\end{align}
We will prove the upper bound \eqref{A12-A32} for the first term on the right hand side of the first line in \eqref{bound for A23}. The other four terms are done in very similar way. 

First we split the integral as
\begin{align}
\label{split-2}
\int_{\R^3}
G^{(k)}_{\ueps} (B_{x, 2r}) \wh H^{(k)}_{\ueps} (dx)
=
\int_{\abs{x}<1}
G^{(k)}_{\ueps} (B_{x, 2r}) \wh H^{(k)}_{\ueps} (dx)
+
\int_{\abs{x}\ge 1}
G^{(k)}_{\ueps} (B_{x, 2r}) \wh H^{(k)}_{\ueps} (dx)
\end{align}
and note that due to \eqref{grbo-loc} and \eqref{grboo-glob} the second term on the right hand side is bounded as 
\begin{align}
\label{second term on right}
\int_{\abs{x}\ge 1}
G^{(k)}_{\ueps} (B_{x, 2r}) \wh H^{(k)}_{\ueps} (dx)
\le 
C r^4. 
\end{align} 
To bound the first term on the right hand side of \eqref{split-2} we proceed as follows
\begin{align}
\notag 
\int_{\abs{x}<1}
G^{(k)}_{\ueps} (B_{x, 2r}) \wh H^{(k)}_{\ueps} (dx)
&
\le 
C
\int_0^1 
F(u) d \wh H^{(k)}_{\ueps} (B_{0,u})
\\
\notag
&
=
Cr^3 \wh H^{(k)}_{\ueps} (B_{0,1})
-
C 
\int_0^1 
\wh H^{(k)}_{\ueps} (B_{0,u}) F^{\prime}(u)du
\\
\notag
&
\le 
C r^4
+
C r^4
\int_r^1
u^{-2}\abs{\log u}du
\\
\label{first term on right}
&
\le 
Cr^4 + Cr^3 \abs{\log r}.
\end{align} 
In the first step we have used \eqref{grbo-loc}. The second step is an integration by parts. In the third step we use \eqref{grbooooooo-loc}, \eqref{grboo-glob} and the explicit form of the function $F$. The last step is explicit integration. 

Finally, \eqref{split-2}, \eqref{second term on right}, \eqref{first term on right} and identical computations for the second term on the right hand side of the first line in \eqref{bound for A23} yield the first inequality in \eqref{A12-A32}. The second line of \eqref{A12-A32} for $b=3$ is proved in an identical way, which we omit to repeat. The case $b=1$ is done in a formally identical way. 

Finally, \eqref{etaone} follows from \eqref{A11-A33-A13}, \eqref{A22} and \eqref{A12-A32}. 

\qed

\section{Proof of Theorem \ref{thm:tricky} -- concluded}
\label{s: End of proof}

As in section \ref{ss: Multi-leg concatenation} let $\varpi_n = \left(\gamma_n; \;  (\xi_{n,j},u_{n,j}): 1\le j\le \gamma_n\right)$, $n\ge 1$,  be a sequence of i.i.d \emph{packs}. Denote $\theta_n$, $((Y_n(t),Z_n(t)): 0\le t\le \theta_n)$ the pair of $Y$ and (forward) $Z$-processes constructed from them and 
\begin{align*}
Y(t)= \sum_{k=1}^{\nu_t} Y(\theta_k)+Y_{\nu_t+1}(\{t\}), 
\qquad
Z(t)= \sum_{k=1}^{\nu_t} Z(\theta_k)+Z_{\nu_t+1}(\{t\}). 
\end{align*}
Beside these two we now define yet another auxiliary process $t\mapsto \cX(t)$ as follows: \\
$\left(\cX_n(t): 0\le t \le \theta_n\right)$ is the Lorentz exploration process constructed with data from 
\\
$\left(Y_n(t): 0\le t \le \theta_n\right)$ \emph{and} incoming velocity 
\begin{align*}
u_{n,0}=
\begin{cases}
u_0 & \text{ if } n=1, 
\\
\dot\cX_{n-1}(\theta_{n-1}^-) & \text{ if } n>1. 
\end{cases}
\end{align*} 
Finally, from these legs concatenate 
\begin{align*}
\cX(t) = \sum_{k=1}^{\nu_t} \cX(\theta_k)+\cX_{\nu_t+1}(\{t\}). 
\end{align*}
Note that the auxiliary process $\left(\cX(t): 0\le t <\infty\right)$ is not identical with the Lorentz exploration process $\left(X(t):0\le t <\infty\right)$, constructed with data from $\left(Y(t): 0\le t \le \infty\right)$ and initial incoming velocity $u_0$, since the former one does not take into account memory effects caused by earlier legs. However, based on Propositions \ref{prop: Z=X in one leg} and \ref{prop:no-interference-between-legs}, we will prove that until time $T=T(r)=\ordo(r^{-2}\abs{\log r}^{-2})$ the processes $t\mapsto X(t)$, $t\mapsto \cX(t)$, and $t\mapsto Z(t)$ coincide with high probability. 

For this, we define the (discrete) stopping times
\begin{align*}
&
\rho
:=
\min\{n: \cX_n(t)\not\equiv Z_n(t), 0\le t \le \theta_n\}
\\
&
\sigma
:=
\min\{n: \max\{\one_{\wt W_n}, \one_{\wh W_n}>0\}=1\}, 
\end{align*}
and note that by construction
\begin{align*}
\inf\{t: Z(t)\not= X(t)\}\ge\Theta_{\min\{\rho,\sigma\}-1}.
\end{align*}

\begin{lemma}
\label{lem: Theta is large}
Let $T=T(r)$ such that $\lim_{r\to\infty} T(r) =\infty$ and $\lim_{r\to\infty} r^{2}\abs{\log r}^2 T(r) =0$. Then 
\begin{align}
\label{Theta is large}
\lim_{r\to 0}
\probab{\Theta_{\min\{\rho,\sigma\}-1}<T}=0. 
\end{align}
\end{lemma}

\begin{lemma}
\label{lem: Z is close to Y}
Let $T=T(r)$ such that $\lim_{r\to\infty} T(r) =\infty$ and $\lim_{r\to\infty} r^{2} T(r) =0$. Then for any $\delta>0$
\begin{align}
\label{Z is close to Y}
\lim_{r\to 0}
\probab{\max_{0\le t \le T}\abs{Y(t)-Z(t)}>\delta \sqrt{T}}=0. 
\end{align}
\end{lemma}

\noindent
{\bf Remark:}
Actually, \eqref{Z is close to Y} holds under the much weaker condition $\lim_{r\to\infty} r \log\log T =0$. This can be achieved by applying the LIL rather than a WLLN type of argument to bound 
\\
$\max_{0\le t \le T}\abs{Y(t)-Z(t)}$ in the proof of Lemma \ref{lem: Z is close to Y}, below. However, since the condition of Lemma \ref{lem: Theta is large} can not be much relaxed, in the end we would not gain much with the extra effort. 

\begin{proof}
[Proof of Lemma \ref{lem: Theta is large}]
\begin{align}
\notag
\probab{\Theta_{\min\{\rho,\sigma\}-1}<T}
&
\le 
\probab{\rho\le 2 \expect{\theta}^{-1}T}
+
\probab{\sigma\le 2 \expect{\theta}^{-1}T}
+
\probab{\sum_{j=1}^{2 \expect{\theta}^{-1}T} \theta_j<T}
\\
\label{trivi-1}
&
\le
Cr^2\abs{\log r}^2 T
+
Cr^2 T
+
C e^{-c T}, 
\end{align}
where $C<\infty$ and $c>0$. The first term on the right hand side of \eqref{trivi-1} is bounded by union bound and \eqref{Z neq X in one leg} from Proposition \ref{prop: Z=X in one leg}. Likewise, the second term is bounded  by union bound and \eqref{no-interference-between-legs} of Proposition \ref{prop:no-interference-between-legs}. In bounding the third term we use a  large deviation upper bound for the sum of independent $\theta_j$-s. 

Finally, \eqref{Theta is large} readily follows from \eqref{trivi-1}. 
\end{proof}

\begin{proof}
[Proof of Lemma \ref{lem: Z is close to Y}]
Note first that 
\begin{align*}
\max_{0\le t \le T}\abs{Y(t)-Z(t)}
\le 
\sum_{j=1}^{\nu_T+1}\eta_j\xi_j, 
\end{align*}
with $\nu_T$ and $\eta_j$ defined in \eqref{tau-and-nu-for-Y}, respectively, \eqref{eta}. Hence, 
\begin{align}
\notag
\probab{\max_{0\le t \le T}\abs{Y(t)-Z(t)}>\delta \sqrt{T}}
&
\le 
\probab{\sum_{j=1}^{2T}\eta_j\xi_j>\delta \sqrt{T}}
+
\probab{\nu_T>2T}
\\
\label{trivi-2}
&
\le 
C\delta^{-1}\sqrt{T}r + e^{-cT}, 
\end{align}
with $C<\infty$ and $c>0$. 
The first term on the right hand side of \eqref{trivi-2} is bounded by Markov's inequality and the straightforward bound
\begin{align*}
\expect{\eta_j\xi_j}\le C r. 
\end{align*} 
The bound on the second term  follows from a straightforward large deviation estimate on $\nu_T\sim POI(T)$. 

Finally, \eqref{Z is close to Y} readily follows from \eqref{trivi-2}. 

\end{proof}

\eqref{tricky-close} is direct consequence of Lemmas \ref{lem: Theta is large} and \ref{lem: Z is close to Y} and this concludes the proof of  Theorem \ref{thm:tricky}. 
\qed

\section*{Acknowledgements}

We thank Jens Marklof for comments on the first version of this paper. We also thank an anonymous referee for their thorough and detailed review which helped us produce a much improved revised version. 
BT thanks the kind hospitality of the Isaac Newton Institute, Cambridge, where part of this work was completed during the  Fall 2018 program `Scaling limits, rough paths, quantum field theory'.
CL was supported by EPSRC Studentship EP/N509619/1 1793795.
The work of BT was supported by EPSRC (UK) Fellowship EP/P003656/1 and by NKFI (HU) K-129170.

\vskip2cm

\hbox{
\hskip9cm
\vbox{\hsize=7cm\noindent
{\sc Authors' address:}
\\
School of Mathematics
\\
University of Bristol
\\
Bristol, BS8 1TW
\\
United Kingdom
\\
{\tt chris.lutsko@bristol.ac.uk}
\\
{\tt balint.toth@bristol.ac.uk}
}
}

\end{document}